\documentclass[reqno]{article}

\usepackage[utf8]{inputenc}
\usepackage[T1]{fontenc}

\usepackage[english]{babel}

\usepackage[margin=1.2in]{geometry}
\usepackage{enumitem}
\usepackage[bookmarks = true]{hyperref}
\usepackage[all]{xy}
\usepackage{xcolor}
\usepackage{amsmath, amsthm, amssymb, bbm}
\usepackage[abbrev]{amsrefs}
\usepackage{unicode-math}
\usepackage{varwidth}

\newtheorem{thm}{Theorem}[section]
\newtheorem*{thm*}{Theorem}
\newtheorem{lem}[thm]{Lemma}
\newtheorem{prop}[thm]{Proposition}
\newtheorem{cor}[thm]{Corollary}

\theoremstyle{definition}
\newtheorem{rmk}[thm]{Remark}

\newtheorem{defn}[thm]{Definition}
\newtheorem{ex}[thm]{Example}
\newtheorem{conj}[thm]{Conjecture}

\numberwithin{equation}{section}

\newcommand{\mr}{\mathrm}

\newcommand{\wt}{\widetilde}

\newcommand{\Res}{\operatorname{Res}}

\newcommand{\Tor}{\operatorname{Tor}}

\newcommand{\Supp}{\operatorname{Supp}}

\newcommand{\Int}{\operatorname{Int}}

\newcommand{\Aut}{\operatorname{Aut}}

\newcommand{\Hom}{\operatorname{Hom}}
\newcommand{\Ext}{\operatorname{Ext}}
\newcommand{\End}{\operatorname{End}}

\newcommand{\Lie}{\operatorname{Lie}}

\newcommand{\Spec}{\operatorname{Spec}}

\newcommand{\Spf}{\operatorname{Spf}}

\newcommand{\codim}{\operatorname{codim}}

\newcommand{\GL}{\operatorname{GL}}

\newcommand{\Stab}{\operatorname{Stab}}

\newcommand{\Inv}{\mathrm{Inv}}

\newcommand{\tensor}{\otimes}

\newcommand{\iso}{\cong}

\newcommand{\mbC}{\mathbb{C}}

\newcommand{\mbF}{\mathbb{F}}

\newcommand{\mbN}{\mathbb{N}}

\newcommand{\mbQ}{\mathbb{Q}}
\newcommand{\mbR}{\mathbb{R}}

\newcommand{\mbX}{\mathbb{X}}

\newcommand{\mbZ}{\mathbb{Z}}

\newcommand{\mcF}{\mathcal{F}}

\newcommand{\mcH}{\mathcal{H}}

\newcommand{\mcL}{\mathcal{L}}

\newcommand{\mcO}{\mathcal{O}}

\newcommand{\mcS}{\mathcal{S}}

\newcommand{\mcZ}{\mathcal{Z}}

\long\def\[#1]#2{ \begin{#1}#2\end{#1}}
\makeatletter
\def\m{\@tut\@gobble\@tutt}
\def\@tut#1\@tutt#2{
\@ifnextchar,{\@tut #1 & #2\\\expandafter\expandafter\expandafter\@gobble\expandafter\@gobble\@gobble\@tutt}{
\@ifnextchar.{\begin{pmatrix}#1 &#2\end{pmatrix}\@gobble}{\@tut #1 & #2\@tutt}
}
}
\def\bm{\@tuut\@gobble\@tuutt}
\def\@tuut#1\@tuutt#2{
\@ifnextchar,{\@tuut #1 & #2\cr\expandafter\expandafter\expandafter\@gobble\expandafter\@gobble\@gobble\@tuutt}{
\@ifnextchar.{\bordermatrix{#1 &#2\cr}\@gobble}{\@tuut #1 & #2\@tuutt}
}
}
\makeatother

\newcommand{\del}{\operatorname{\partial Orb}}

\newcommand{\ov}{\overline}
\newcommand{\lr}{\longrightarrow}

\begin{document}

\title{On the Linear AFL: The Non-Basic Case}
\author{Qirui Li and Andreas Mihatsch}
\date{March 14, 2024}

\maketitle
%
\setlength{\parskip}{-1mm}
\tableofcontents
\setlength{\parskip}{0mm}

\newcommand{\charred}{\mr{charred}}
\newcommand{\Cent}{\mr{Cent}}
\newcommand{\back}{\backslash}
\newcommand{\std}{\mr{std}}
\newcommand{\ord}{\mr{ord}}
\newcommand{\simto}{\overset{\sim}{\to}}
\newcommand{\lsimto}{\overset{\sim}{\lr}}
\newcommand{\Ha}{\mr{Ha}}

\section{Introduction}
\label{s:intro}
The purpose of this article is to formulate a linear arithmetic fundamental lemma conjecture also for non-basic isogeny classes. Our main result is a reduction of these non-basic cases to the basic one. The context and relevance of these results are as follows.

In general, arithmetic fundamental lemmas are certain identities of intersection numbers on moduli spaces of $p$-divisible groups (RZ spaces) and derivatives of orbital integrals. They are, in particular, local statements, formulated over a $p$-adic local field. Their motivation is global, however, namely they arise from a relative trace formula comparison approach to certain intersection problems of special cycles on Shimura varieties. We refer to the original paper of W. Zhang \cite{Z} or his surveys \cite{Z_survey_12, Z_survey_18} for an explanation of this global-to-local formalism. 

There are currently two families of such conjectures. The first comes from the Gan--Gross--Prasad setting of a diagonal embedding of unitary groups $U(V)\to U(V)\times U(V\oplus 1)$. The original conjecture in this context \cite{Z} has been proved in work of W. Zhang \cite{Z19}, of W. Zhang and the second author \cite{M_loc_const, MZ} and of Z. Zhang \cite{ZZ}. Several variants have been formulated by Rapoport--Smithling--Zhang \cite{RSZ1, RSZ2, RSZ3}, Y. Liu \cite{Liu}, W. Zhang \cite{Z_Bessel}, and Z. Zhang \cite{ZZ}. In all these cases, the global situation is such that intersection takes place in the basic locus. Correspondingly, all the cited works are concerned with basic RZ spaces.

The second family arises from the unitary variant of the Guo--Jacquet restriction problem \cite{Guo}. That is, the special cycles in question come from an embedding $U_E(V)\to U_F(V)$, where $V$ is a hermitian space for some field extension $F/F_0$ with an additional action of a quadratic extension $E/F$. At inert places of $F/F_0$, this leads to an AFL for basic unitary RZ spaces that is the content of the forthcoming work of Leslie--Xiao--Zhang \cite{LXZ}. At a split place $v$ however, it leads to intersection problems on RZ spaces for all isogeny classes of dimension $1$ and height $2n$ strict $p$-divisible $O_{F_v}$-modules with $O_{E_v}$-action. These are the non-basic linear AFLs from the title of the paper. We mention that the phenomenon of non-basic intersection at split places also occurs for unitary groups over central simple algebras and refer to \cite{HM} and \cite{Li_Mih} for intersection number identities in such situations.

We now describe our results in some detail. They also pertain to the function field setting. We do not exclude the prime $2$.

Let $E/F$ be an unramified quadratic extension of a non-archimedean local field $F$ with uniformizer $π\in F$. Denote by $\breve F$ the completion of a maximal unramified extension of $F$ and fix an embedding $E\to \breve F$. Let $\mbX$ be a $1$-dimensional, height $2n$, not necessarily formal, strict $π$-divisible $O_F$-module over the residue field $\mbF$ of $\breve F$ (cf. Def. \ref{def:strict_O_F_module}). Additionally, consider a pair of embeddings
$$β = (β_1,β_2),\ \ \ β_i:E\lr \End^0(\mbX).$$
The $π$-divisible $O_F$-module $\mbX$ decomposes in an essentially unique way as $\mbX = \mbX^0\times \mbX^1$ with $\mbX^0$ connected and $\mbX^1$ étale. Since $\End^0(\mbX) = \End^0(\mbX^0)\times \End^0(\mbX^1)$, giving a pair $β$ is the same as giving analogous pairs $β^0,β^1$ for the two factors.

Consider now the RZ space $M$ of strict $π$-divisible $O_F$-modules with quasi-isogeny to $\mbX$. It is a formal scheme; abstractly
\begin{equation}
M \iso \coprod_{\mbZ\,\times\, (GL_{2n^1}(F)/GL_{2n^1}(O_F))}\Spf O_{\breve F}[\![t_1,\ldots,t_{2n-1}]\!],
\end{equation}
where $2n^1$ is the height of the étale factor $\mbX^1$. There are two analogous RZ spaces $Z_1, Z_2$ that parametrize strict $O_E$-modules together with an $O_E$-linear quasi-isogeny to $(\mbX, β_i)$. Forgetting about the $O_E$-action defines closed immersions $Z_1,Z_2\to M$ that allow to view them as cycles in middle dimension. Assuming that $(β_1,β_2)$ is regular semi-simple (cf. Def. \ref{def:reg_ss}), the centralizer of the image $β_1(E) \cup β_2(E)$ is an étale $F$-algebra $L \subseteq \End^0(\mbX)$ of degree $n$. Write $L = \prod_{j\in J} L^j$ as product of fields, pick uniformizers $π_j\in L^j$ and define $Γ = \prod_{j\in J} π_j^\mbZ \subset L^\times$. Then $L^\times$ acts compatibly on $Z_1, Z_2$ and $M$, and the quotient $Γ\backslash (Z_1 \cap Z_2)$ is artinian. We define
\begin{equation}
\mr{Int}(β) := \ell_{O_{\breve F}} (\mcO_{Γ\backslash (Z_1\cap Z_2)}).
\end{equation}
To $β$ one may also associate a matching pair (cf. Def. \ref{def:matching})
$$α = (α_0, α_3),\ \ \ α_i:F\times F \lr M_{2n}(F),$$
which is unique up to conjugation. Given an element $f\in \mcH = \mbC[GL_{2n}(O_F)\backslash GL_{2n}(F)/GL_{2n}(O_F)]$ of the spherical Hecke algebra, one may define an orbital integral for the $GL_n(F\times F)\times GL_n(F\times F)$-action on $GL_{2n}(F)$,
\begin{equation}
O(α, f, s),\ \ \ s\in \mbC.
\end{equation}
This is an element of $\mbC[q_F^s, q_F^{-s}]$, where $q_F$ is the residue cardinality of $F$. Under our assumption that $α$ matches $β$, the central value $O(α,f,0)$ vanishes.
\begin{conj}[Linear AFL]\label{conj:intro}
Assume $α$ and $β$ to be as above. Then there is an equality,
$$\left.\frac{1}{\log(q_F)}\frac{d}{ds}\right\vert_{s = 0} O(α, 1_{GL_{2n}(O_F)}, s) = \mr{Int}(β).$$
\end{conj}
The basic case is precisely when $\mbX = \mbX^0$ is connected. In this situation, the above conjecture has (up to sign) already been formulated in work of the first author \cite{Li} and in his joint work with Howard \cite{HL}. In fact, \cite{Li} gives a definition of $\mr{Int}(β, f)$ for every $f\in \mcH$ and formulates Conj. \ref{conj:intro} in this generality, while \cite{HL} allows the cycles $Z_1,Z_2$ to be defined for two different quadratic extensions $E_1,E_2/F$ (biquadratic setting). We work in the combined generality throughout the paper and refer to Conj. \ref{conj:AFL} below for the general version.

Conj. \ref{conj:intro} is known when $n = 1$ or $2$. The more general version, Conj. \ref{conj:AFL}, is known in all cases when $n=1$, cf. \cite{Li, HL}. It is currently only known for the unit function $f = 1_{GL_{2n}(O_F)}$ (but possibly $E_1 \not \cong E_2$) when $n = 2$, cf. \cite{Li_GL4, Li_future}.

The following two are our main results.
\begin{thm}\label{thm:intro}
Write $β = (β^0, β^1)$ for the two components of $β$ with respect to $\mbX = \mbX^0\times \mbX^1$. Then Conj. \ref{conj:intro} holds for $β$ if and only if its hold for the connected component $β^0$.
\end{thm}
\begin{cor}\label{cor:intro}
The AFL (Conj. \ref{conj:intro}) holds for all isogeny classes $\mbX$ whose connected component $\mbX^0$ has height $\leq 4$.
\end{cor}
We formulate and prove these two result also in the biquadratic case and for general spherical Hecke functions, cf. Thm. \ref{thm:main} and Cor. \ref{cor:main}. In this generality, however, we have to assume the Guo--Jacquet Fundamental Lemma (resp. its biquadratic variant) for the whole spherical Hecke algebra; it is currently only known for the unit Hecke function and $E_1 = E_2$, see \cite{Guo}.

The proof of Thm. \ref{thm:intro} relies on the connected-étale sequence $0\to X^0\to X\to X^1\to 0$ of the universal $π$-divisible $O_F$-module $X$ over $M$. It provides a fibration map
\begin{equation}
\label{eq:fibration_intro}
M \lr M^0\times_{\Spf O_{\breve F}} M^1
\end{equation}
to the product of the RZ spaces for $\mbX^0$ and $\mbX^1$. If $Y^0$ and $Y^1$ denote their universal $π$-divisible $O_F$-modules, then $M$ can be identified with the space of torsion extensions $\underline{\Ext}^1(Y^1, Y^0)_{\mr{tors}}$. This space is in bijection with the Hom-functor $\underline{\Hom}(T(Y^1), Y^0)$, where $T(Y^1)$ denotes the Tate module. In this way, $M$ is described as the total space of a $π$-divisible $O_F$-module over $M^0\times_{\Spf O_{\breve F}} M^1$. This description allows to relate intersections in the three spaces and we prove the identity
\begin{equation}\label{eq:int_number_identity_intro}
\mr{Int}(β) = |\mr{Res}(\mr{Inv}(β^0), \mr{Inv}(β^1))|^{-1}_F\, O(β^1, 1_{GL_{2n^1}(O_F)})\, \mr{Int}(β^0).
\end{equation}
Here, $\Inv(β^0)$, $\Inv(β^1)\in F[T]$ denote the invariants of $β^0$ and $β^1$ as in Rmk. \ref{rmk:inv_E}, $\mr{Res}(\cdot,\cdot)$ denotes the resultant of two polynomials, and $|\cdot|_F$ denotes the normalized absolute value on $F$. A similar identity can be shown for $O(α, f, s)$, leading to the proof of Thm. \ref{thm:intro}. Again, we carry out all these ideas in the biquadratic setting and for arbitrary $f$.

The structure of this article is as follows. In \S\ref{s:FL}, we recall necessary background on matching, orbital integrals and the Guo--Jacquet Fundamental Lemma. This is essentially taken from \cite{HL}, our own addition being a formulation of the analytic side of Thm. \ref{thm:intro}. In \S\ref{s:AFL}, we define our intersection numbers and formulate the AFL. \S\ref{s:reduction} is devoted to the proof of Thm. \ref{thm:intro}, its heart being the fibration argument for \eqref{eq:fibration_intro} in \S\ref{ss:fibration_RZ} and a linear algebra computation that produces the resultant factor in \eqref{eq:int_number_identity_intro}. Finally, our appendices \S\ref{s:appendix_loc_inter} and \S\ref{s:appendix_Satake} provide background material on correspondences and the partial Satake transform.

\subsection*{Acknowledgments}

We thank M. Rapoport for his continued interest in our work and several helpful discussions. We are grateful to U. Hartl for correspondence on $π$-divisible groups. We furthermore thank N. Hultberg, W. Zhang and Z. Zhang for comments on an earlier version of this article. We also thank the referee for their suggestions.

\section{The Fundamental Lemma}
\label{s:FL}
This section recalls necessary background on matching, orbital integrals and the fundamental lemma from \cite{Guo} and \cite{HL}. 

\subsection{Invariants}

Fix a field $F$ and an integer $n\geq 1$. We consider two quadratic étale extensions $E_1,E_2/F$ and write $σ_1$ resp. $σ_2$ for their non-trivial involutions over $F$. Let $B = \prod_{j\in J} B^j$ be the product of a finite family $(B^j)_{j\in J}$ of central simple $F$-algebras. We assume that the degree $[B^j:F] = \dim_F(B^j)^{1/2}$ of each factor is even and that the total degree $[B:F] = \sum_{j\in J} [B^j:F]$ is $2n$.

We write $β:(E_1,E_2)\to B$ for pairs $(β_1:E_1\to B$, $β_2:E_2\to B)$ of embeddings as $F$-algebras into $B$. Giving a map $β_i:E_i\to B$ is the same as giving a tuple of maps $β_i^j:E_i\to B^j$. Without further mentioning, we only consider maps $β_i^j$ that make $B^j$ into a free $E_i$-module. This additional specification only matters if $E_i \iso F\times F$. In particular, pairs $β:(E_1, E_2)\to B$ are the same as tuples of pairs $β^j:(E_1, E_2)\to B^j$.

Let $E_3 \subset E_1\tensor_F E_2$ be the fixed ring under $σ_1\tensor σ_2$. It is a quadratic étale $F$-algebra and we denote by $σ_3:E_3\to E_3$ its non-trivial $F$-involution. Given $β$ as above, Howard--Li \cite{HL}*{\S2} define a canonical element $\mathbf{s}_β \in E_3\tensor_F B$ that centralizes $β_1(E_1)$ and $β_2(E_2)$. Their definition, which we do not need here, extends verbatim from CSAs to products of CSAs; in fact, simply $\mathbf{s}_β = (\mathbf{s}_{β^j})_{j\in J}$. They show in \cite{HL}*{Prop. 2.2.3} that the reduced characteristic polynomial
$$\mr{charred}_{E_3 \tensor_F B/E_3}(\mathbf{s}_β; T) = \prod_{j\in J} \charred_{E_3\tensor_FB/E_3}(\mathbf{s}_{β^j}; T) \in E_3[T]$$
is a square and make the following definition.

\begin{defn}\label{def:reg_ss}
\begin{enumerate}[wide, labelindent=0pt, labelwidth=!, label=(\arabic*), topsep=2pt, itemsep=2pt]
\item The set of invariants of degree $n$ is the set of monic polynomials $δ\in E_3[T]$ of degree $n$ that satisfy the symmetry
\begin{equation}\label{eq:def_symmetry}
(-1)^n δ(1-T) = σ_3(δ(T)).
\end{equation}
An invariant $δ$ is called regular semi-simple if for every map $ρ:E_3 \to \ov{F}$, the polynomial $ρ(δ)$ has $n$ distinct non-zero roots.
\item The invariant $\Inv(β) \in E_3[T]$ of a pair $β$ is the unique monic square root of $\mr{charred}_{E_3\tensor_F B/E_3}(\mathbf{s}_β; T).$ It satisfies \eqref{eq:def_symmetry} by \cite[Prop. 2.3.3]{HL}. A pair $β$ is called regular semi-simple if its invariant is regular semi-simple.
\end{enumerate}
\end{defn}

\begin{rmk}\label{rmk:inv_E}
If $E_1 = E_2$, then $E_3 \iso F\times F$. Fixing such an isomorphism, one may renormalize the invariant by considering only the first component of $\Inv(β;T) \in (F\times F)[T]$. This was done for the formulation of \eqref{eq:int_number_identity_intro}.
\end{rmk}

Note that $\Inv(β) = \prod_{j\in J} \Inv(β^j)$ and that $\Inv(β)$ only depends on the $B^\times$-conjugation orbit of $β$. For regular semi-simple $β$ and simple $B$, a converse holds:
\begin{prop}[\protect{\cite[Cor. 2.5.7]{HL}}]\label{prop:regular_semi_simple}
Assume that $B$ is a central simple $F$-algebra. Two regular semi-simple pairs $β,β':(E_1,E_2)\to B$ are $B^\times$-conjugate if and only if $\mr{Inv}(β) = \mr{Inv}(β')$.
\end{prop}

\begin{rmk}
The converse direction of Prop. \ref{prop:regular_semi_simple} does not hold for non-simple $B$. For example, take $E_1 = E_2 = F\times F$ and consider two regular semi-simple pairs $β^0, β^1:(E_1, E_2)\to M_2(F)$ with $\Inv(β^0) \neq \Inv(β^1)$. Then
$$(β^0, β^1),\ (β^1, β^0):(E_1, E_2)\lr B := M_2(F)\times M_2(F)$$
are both regular semi-simple with invariant $\Inv(β^0)\Inv(β^1)$. They are not $B^\times$-conjugate, however, because $β^0$ and $β^1$ are not $GL_2(F)$-conjugate. 
\end{rmk}

An immediate question is, given $E_1$, $E_2$, $B$ and a regular semi-simple invariant $δ$ as above, whether or not there exists a pair $β:(E_1, E_2) \to B$ with $\Inv(β) = δ$. This is partially answered by the following universal construction.

\begin{prop}[\cite{HL}*{Prop. 2.5.6}]\label{prop:universal_pair}
Given $E_1$, $E_2$ and a regular semi-simple invariant $δ\in E_3[T]$ of degree $n$, there exists an étale $F$-algebra $L_δ$ of degree $n$ and a quaternion algebra $B_δ/L_δ$, together with a pair of embeddings $β_δ:(E_1,E_2)\to B_δ$, that has the following properties.
\begin{enumerate}[wide, labelindent=0pt, labelwidth=!, label=(\arabic*), topsep=2pt, itemsep=2pt]
\item The image $β_{δ,1}(E_1) \cup β_{δ,2}(E_2)$ generates $B_δ$ as $F$-algebra.
\item The natural maps $β_{δ,i}(E_i)\tensor_F L_δ\to B_δ$ are injective and make $β_{δ,i}(E_i)\tensor_FL_δ$ into a maximal commutative $F$-subalgebra of $B$.
\item For every product $B$ of central simple $F$-algebras of total degree $[B:F] = 2n$, and for every pair $β:(E_1,E_2)\to B$ of invariant $δ$, there is a unique embedding of $F$-algebras $ι:B_δ\to B$ such that $β = ι\circ β_δ$. Conversely, if $β = ι\circ β_δ$ for an embedding $ι:B_δ\to B$ of $F$-algebras, then $\Inv(β) = δ$.
\end{enumerate}
\end{prop}
\begin{proof}
Starting from $(E_1, E_2, δ)$, the datum $β_δ:(E_1, E_2)\to B_δ$ is constructed during the proof of \cite{HL}*{Prop. 2.5.6}. (The algebra $B_δ$ is denote by $\mcF(φ_1, φ_2)$ there.) It has properties (1) and (2) by construction. Property (3) is stated in \cite{HL}*{Prop. 2.5.6} for simple $B$. The statement for general $B$ is obtained by the same proof.
\end{proof}

Thus pairs $β$ of invariant $δ$ for $B$ are the same as $F$-algebra maps $B_δ\to B$. Implicit in Prop. \ref{prop:universal_pair} is moreover the statement that the centralizer of a regular semi-simple pair $β$ is an étale $F$-algebra of degree $n$. Namely, it is isomorphic to the center $L_δ\subset B_δ$. By the construction of $B_δ$ in \cite{HL}, it may be described (up to isomorphism) by
\begin{equation}\label{eq:descript_L}
L_δ \iso (E_3[T]/(δ(T)))^{σ_3 = \mr{id}}
\end{equation}
where $σ_3$ is extended to $E_3[T]$ as $σ_3(T) = 1-T$.

\subsection{Matching}
From now on, we let $F$ be a non-archimedean local field with normalized absolute value $|\cdot|$ and assume $E_1,E_2$ to be étale field extensions.
Put $E_0 = F\times F$. Then the invariant $\Inv(α)$ of a pair $α = (α_0, α_3):(E_0, E_3)\to M_{2n}(F)$ lies in $E_3[T]$ as well.
\begin{defn}\label{def:matching}
Two regular semi-simple pairs $α:(E_0,E_3)\to M_{2n}(F)$ and $β:(E_1,E_2)\to B$ are said to match if $\Inv(α) = \Inv(β)$.
\end{defn}

Let $δ\in E_3[T]$ be regular semi-simple and of degree $n$. Let $α_δ:(E_0, E_3) \to A_δ$ be the universal datum for $(E_0, E_3, δ)$ from Prop. \ref{prop:universal_pair} and let $L_δ = \Cent(A_δ)$. By part (2) of Prop. \ref{prop:universal_pair}, the $L_δ$-algebra $E_0\tensor_FL_δ \iso L_δ\times L_δ$ embeds into $A_δ$ which implies $A_δ \iso M_2(L_δ)$. In particular, there exists an embedding $ι:A_δ \to M_{2n}(F)$. By Prop. \ref{prop:universal_pair} (3), $\Inv(ι\circ α_δ) = δ$. It follows that for every regular semi-simple invariant $δ$ of degree $n$, there exists a pair $α:(E_0, E_3)\to M_{2n}(F)$ with $\mr{Inv}(α) = δ$. On the other hand, given $E_1$, $E_2$, $B$ and $δ$, there need not exist a pair $β:(E_1, E_2)\to B$ with $\mr{Inv}(β) = δ$. Def. \ref{def:matching} is asymmetric in this sense.

\subsection{Orbital integrals for $(E_1, E_2)$}
The fundamental lemma compares $η$-twisted orbital integrals of pairs $(E_0,E_3)\to M_{2n}(F)$ with orbital integrals of matching pairs $(E_1,E_2)\to M_{2n}(F)$. In this section, we define the second kind. Throughout, $Λ_{\mr{std}} := O_F^{2n}$ denotes the standard lattice and $\mcH = \mbC[GL_{2n}(O_F)\backslash GL_{2n}(F) / GL_{2n}(O_F)]$ the spherical Hecke algebra.

\begin{defn}\label{def:orb_int_geom_side}
Let $β:(E_1,E_2)\to M_{2n}(F)$ be regular semi-simple and write $H_i := \mr{Cent}(E_i)^\times$. Choose $g_i \in GL_{2n}(F)$ such that $g_iΛ_{\mr{std}}$ is $O_{E_i}$-stable. For $f\in \mcH$, define
\begin{equation}\label{eq:def_orb_int_geom_side}
O(β, f) := \int_{(H_1\cap H_2)\back (H_1\times H_2)} f(g_1^{-1}h_1^{-1}h_2g_2) dh_1dh_2.
\end{equation}
\end{defn}
Here we endow $H_i \iso GL_n(E_i)$ with the Haar measure such that $\mr{Vol}(GL_n(O_{E_i})) = 1$. The stabilizer $H_1\cap H_2$ is the group of units $L^\times$, where $L = \Cent(β(E_1)\cup β(E_2)) \iso L_δ$ is an étale $F$-algebra of degree $n$ by Prop. \ref{prop:universal_pair}. We endow it with the measure such that $\mr{Vol}(O_L^\times) = 1$. The definition is independent of the chosen $g_i$, once convergence is established:

\begin{lem}
The subset $(H_1\cap H_2)\back (H_1\times H_2)\subset GL_{2n}(F)$ is closed. In particular, the integrand in \eqref{eq:def_orb_int_geom_side} has compact support.
\end{lem}
\begin{proof}
Set $G = GL_{2n}(F)$. There is an isomorphism
$$G\back (G/H_1\times G/H_2) \lsimto H_1\back G/H_2,\ \ (g_1, g_2)\longmapsto g_1^{-1}g_2.$$
By Skolem--Noether, via $(g_1, g_2)\mapsto (g_1β_1g_1^{-1}, g_2β_2g_2^{-1})$, the source is in bijection with the $G$-orbits of pairs $(E_1, E_2)\to M_{2n}(F)$. By Prop. \ref{prop:regular_semi_simple}, the orbit of $β$ agrees with the locus where the invariant takes value $\Inv(β)$. It is, in particular, closed. Since the orbit of $β$ maps to $H_1H_2$ in the target, it follows that $(H_1\cap H_2)\back (H_1\times H_2)\subset GL_{2n}(F)$ is a closed subset as claimed.
\end{proof}

The orbital integral $O(β, f)$ can also be understood as an $f$-weighted lattice count. Let $\mcL(β_i)$ denote the $O_{E_i}$-stable lattices in $F^{2n}$. This set is stable under the $L^\times$-action because $L^\times$ centralizes $β_i(E_i)$. Let $g_i\in GL_{2n}(F)$ be fixed as in Def. \ref{def:orb_int_geom_side} and set $K_i = \Stab(g_iΛ_\std)\cap H_i$. Then
\begin{equation}\label{eq:L_beta}
H_i/K_i \lsimto \mcL(β_i),\ \ h_i \longmapsto h_ig_iΛ_\std.
\end{equation}
As $f$ lies in the spherical Hecke algebra, the restriction $f\vert_{g_1^{-1}K_1h_1^{-1} h_2K_2g_2}$ is constant for every element $(h_1, h_2)\in H_1\times H_2$. Moreover, the volume of $L^\times \back (L^\times \cdot (h_1K_1\times h_2K_2))$ is equal to
\begin{equation}\label{eq:L_volume}
\mr{Vol}(L^\times \cap h_1K_1 h_1^{-1}\cap h_2K_2 h_2^{-1}) = \mr{Vol}(\mr{Stab}_{L^\times}(h_1g_1Λ_\std, h_2g_2Λ_\std)).
\end{equation}
We introduce the following notation: Given two $O_F$-lattices $Λ_1,Λ_2\subset F^{2n}$, we write $f(Λ_1,Λ_2)$ for the value of $f$ at the relative position of $Λ_1$ and $Λ_2$. That is, $f(Λ_1,Λ_2) := f(k_1^{-1}k_2)$ where the $k_i\in GL_{2n}(F)$ are chosen with $Λ_i = k_iΛ_{\mr{std}}$. Also choose a cocompact, free, discrete subgroup $Γ\subseteq L^\times$. Combining \eqref{eq:L_beta} and \eqref{eq:L_volume} with Def. \ref{def:orb_int_geom_side}, we obtain
\begin{equation}\label{eq:lattice_count_1}
\begin{aligned}
O(β,f)\ & =\ \sum_{(Λ_1, Λ_2) \in L^\times \backslash (\mathcal{L}(β_1)\times \mathcal{L}(β_2))} \mr{Vol}(\mr{Stab}_{L^\times}(Λ_1,Λ_2))^{-1} f(Λ_1,Λ_2)\\[3mm]
\ & =\ \mr{Vol}(Γ\backslash L^\times)^{-1} \sum_{(Λ_1, Λ_2) \in Γ \backslash (\mathcal{L}(β_1)\times \mathcal{L}(β_2))} f(Λ_1,Λ_2).
\end{aligned}
\end{equation}

\subsection{$η$-Twisted orbital integrals for $(E_0, E_3)$}
\label{ss:eta_orb}
Let $η:E_3^\times \to \{\pm 1\}$ be the character from local class field theory for the extension $E_1\tensor_FE_2/E_3$. Our aim is to define certain $η$-twisted orbital integrals for regular semi-simple pairs $α:(E_0, E_3)\to M_{2n}(F)$. We only consider spherical Hecke functions in this article. For these, the definition of such integrals is only interesting when $η$ restricts to the trivial character on $O_{E_3}^\times$; cf. \cite{HL}*{Prop. 3.2.7}. Thus we assume from now on that $E_1\tensor_FE_2/E_3$ is unramified. Equivalently, we assume that one out of $E_1, E_2$ is unramified over $F$. This leaves us with the following two cases:
\begin{enumerate}[wide, labelindent=0pt, labelwidth=!, label=(\arabic*), topsep=2pt, itemsep=2pt]
\item If $E_1, E_2$ are both unramified over $F$, then $E_3 \iso F\times F$ and $E_1\tensor_FE_2 \iso E_1 \times E_1$ is unramified over $E_3$.
\item If precisely one out of $E_1, E_2$ is ramified over $F$, then $E_1\tensor_FE_2$ is a biquadratic field extension of $F$ such that $σ_1 \tensor σ_2$ acts non-trivially on the residue field. It follows that $E_3/F$ is ramified and $E_1\tensor_FE_2/E_3$ unramified.
\end{enumerate}
In particular, we see that in both cases $η$ can be defined by $η(x) = (-1)^{[O_{E_3}: xO_{E_3}]}$, where the index\footnote{If $Λ' \subset Λ\subset V$ are $O_F$-lattices in a finite dimension $F$-vector space $V$, then we set $[Λ:Λ'] = \log_q|Λ/Λ'|$ where $q = q_F$. For general $Λ', Λ \subset V$, we put $[Λ:Λ'] = [Λ:π^nΛ'] - n\dim(V)$ for $n\gg 0$.} is meant as $O_F$-lattices in $E_3$. We define $η$ to $GL_n(E_3)$ by $η(h) := η(\det h)$. Denote by $|\cdot|_F$ the normalized absolute value on $F$. Recall that $E_0 = F\times F$. We also define the character
\begin{equation}
\begin{aligned}
|\cdot|:GL_n(E_0)\ & \lr\ \mbR^\times\\
(a,b)\ & \longmapsto \ \left|\frac{\det(a)}{\det(b)}\right|.
\end{aligned}
\end{equation}

The definition of $O(β, f)$ in \eqref{eq:def_orb_int_geom_side} required the choice of auxiliary elements $g_1, g_2\in GL_{2n}(F)$. The situation for $η$-twisted orbital integrals is similar, but we need to be more careful to ensure independence of such an auxiliary choice. Consider a regular semi-simple pair $α:(E_0, E_3)\to M_{2n}(F)$. We write $\mcL(α_i)$ for the set of $α_i(O_{E_i})$-lattices in $F^{2n}$. Any lattice $Λ_0\in \mcL(α_0)$ is of the form $Λ_0^+\times Λ_0^-$, where $Λ_0^+\subset α_0(1, 0)F^{2n}$ is an $α_0(O_F\times 0)$-lattice and $Λ_0^-\subset α_0(0,1)F^{2n}$ an $α_0(0\times O_F)$-lattice. Since $α$ is regular semi-simple, one can show (omitted) that the two natural maps obtained from $α_3$,
$$E_3\tensor_F \left(α_0(1,0)\cdot F^{2n}\right) \lr F^{2n},\qquad E_3\tensor_F \left(α_0(0,1)\cdot F^{2n}\right) \lr F^{2n}$$
are bijective. In particular, given $(Λ_0, Λ_3)\in \mcL(α_0)\times \mcL(α_3)$, there are well-defined lattice indices
\begin{equation}
[Λ_3 : O_{E_3} Λ_0^+],\ [Λ_3:O_{E_3} Λ_0^-]\ \in \mbZ.
\end{equation}
\begin{defn}\label{def:orb_int_analytic_side}
Let $α:(E_0,E_3)\to M_{2n}(F)$ be regular semi-simple. Given $(Λ_0, Λ_3)\in \mcL(α_0)\times \mcL(α_3)$ and $s\in \mbC$, the transfer factor of $(Λ_0, Λ_3)$ is defined by
\begin{equation}\label{eq:root_number}
\Omega(Λ_0, Λ_3, s) = (-1)^{[Λ_3:O_{E_3}Λ_0^+]} \cdot \big(q^{[Λ_3:O_{E_3}Λ_0^-] - [Λ_3:O_{E_3}Λ_0^+]}\big)^{s/2}.
\end{equation}
This function has the transformation property
\begin{equation}\label{eq:trafo_root}
\Omega(h_0Λ_0, h_3Λ_3, s) = |h_0|^s η(h_3) \Omega(Λ_0, Λ_3, s),\qquad h_0\in H_0,\ h_3\in H_3.
\end{equation}
Choose $g_0, g_3\in GL_{2n}(F)$ such that $g_iΛ_{\mr{std}} \in \mcL(α_i)$. For $f\in \mcH$, define
\begin{equation}\label{eq:def_orb_int_analytic_side}
O(α, f, s) \ :=\ \Omega(g_0Λ_{\mr{std}}, g_3Λ_{\mr{std}}) \cdot \int_{(H_0\cap H_3)\back (H_0\times H_3)} f(g_0^{-1}h_0^{-1}h_3g_3)|h_0|^s η(h_3) dh_0 dh_3.
\end{equation}
The integrand here is well-defined by \cite[Lem. 3.2.3]{HL} which states that $η$ and $|\cdot |$ are trivial on $H_0\cap H_3$. Moreover, the normalization of Haar measures in \eqref{eq:def_orb_int_analytic_side} is as in Def. \ref{def:orb_int_geom_side} and the same remarks on convergence apply. Identity \eqref{eq:trafo_root} allows to rewrite \eqref{eq:def_orb_int_analytic_side} as
\begin{equation}\label{eq:orb_int_analytic_simple}
O(α, f, s) = \int_{(H_0\cap H_3)\back (H_0\times H_3)} f(g_0^{-1}h_0^{-1}h_3g_3)\, \Omega(h_0g_0Λ_\std, h_3g_3Λ_\std)\, dh_0 dh_3.
\end{equation}
As $f$ is spherical, $f(g_0^{-1}h_0^{-1}h_3g_3)$ only depends on the lattices $g_0Λ_\std$ and $Λ_3Λ_\std$ and not on the precise choices of $g_0, g_3$. Using that $H_i$ acts transitively on $\mcL(α_i)$, we see that $O(α, f, s)$ is actually independent of $g_0$ and $g_3$.
\end{defn}

There is a lattice count formula for $O(α, f, s)$ that is analogous to \eqref{eq:lattice_count_1}: Let $L$ denote the centralizer of $α$ and pick a cocompact subgroup $Γ\subset L^\times$ as before. Then \eqref{eq:orb_int_analytic_simple} gives
\begin{equation}\label{eq:lattice_count_2}
\begin{aligned}
O(α,f,s)\ & =\ \sum_{(Λ_0, Λ_3) \in L^\times \backslash (\mathcal{L}(α_0)\times \mathcal{L}(α_3))} \mr{Vol}(\mr{Stab}_{L^\times}(Λ_0,Λ_3))^{-1} \,f(Λ_0,Λ_3) \,\Omega(Λ_0, Λ_3, s)\\[3mm]
\ & =\ \mr{Vol}(Γ\backslash L^\times)^{-1} \sum_{(Λ_0, Λ_3) \in Γ \backslash (\mathcal{L}(α_0)\times \mathcal{L}(α_3))} f(Λ_0,Λ_3)\,\Omega(Λ_0, Λ_3, s).
\end{aligned}
\end{equation}

\subsection{Fundamental Lemma}
Recall that we assumed from \S\ref{ss:eta_orb} that at least one out of $E_1,E_2$ is unramified over $F$.

\begin{conj}[Guo--Jacquet Fundamental Lemma]\label{conj:FL}
Assume that $α:(E_0,E_3)\to M_{2n}(F)$ and $β:(E_1,E_2)\to M_{2n}(F)$ are regular semi-simple and matching. Let $f\in \mcH$ be a spherical Hecke function. Then
\begin{equation}\label{eq:FL}
O(α, f, 0) = O(β,f).
\end{equation}
\end{conj}

This has been conjectured by Guo--Jacquet \cite{Guo} for $E_1 = E_2$. The biquadratic case ($E_1\neq E_2$) has been conjectured by Howard--Li \cite{HL}. The following cases are known.
\begin{enumerate}[wide, labelindent=0pt, labelwidth=!, label=(\arabic*), topsep=2pt, itemsep=2pt]
\item For $E_1 = E_2$ both unramified and the unit Hecke function $f = 1_{GL_{2n}(O_F)}$, the fundamental lemma is the main result of \cite{Guo}.

\item For general $E_1, E_2$ and $n = 1$, the conjecture is known for the full spherical Hecke algebra, cf. \cite{HL}.

\item For general $E_1, E_2$ and $n = 2$, the conjecture is known for the unit element $1_{GL_{2n}(O_F)}$, which is due to the first author \cite{Li_future}.
\end{enumerate}

\begin{rmk}\label{rmk:FL_open}
FL type identities for the spherical Hecke algebra may sometimes be deduced from the FL for the unit function, using global arguments, which goes back to Clozel \cite{Clozel} and Labesse \cite{Labesse}. This idea is already mentioned in \cite{Guo} for the above FL conjecture, but does not seem to have been brought to fruition yet.
\end{rmk}

\subsection{Reduction formula}
Assume that $n = n^0 + n^1$ and that the two regular semi-simple pairs
$$α:(E_0,E_3)\lr M_{2n}(F),\ \ \ β:(E_1,E_2)\lr M_{2n}(F)$$
factor through $M_{2n^0}(F) \times M_{2n^1}(F)$. Denote their components by $α^0, α^1$ resp. $β^0, β^1$. Our aim is to relate the orbital integrals $O(α,f,s)$ resp. $O(β,f)$ with those for the Levi factors, $O(α^j, f, s)$ resp. $O(β^j, f)$, $j = 0,1$. This relation is due to Guo \cite{Guo} when $E_1 = E_2$.

We work with the following generators (as $\mbC$-vector space) of $\mcH$. For $m\in \mbZ_{\geq 0}$, write
\begin{equation}\label{eq:def_Hecke_simple}
f(m) = 1_{\{g\in M_{2n}(O_F),\ v(\det(g)) = m\}}.
\end{equation}
Given a tuple $m = (m_1,\ldots,m_r)\in \mbZ_{\geq 0}^r$, put
\begin{equation}\label{eq:def_Hecke_iterated}
f(m) = f(m_1)*f(m_2)*\ldots* f(m_r).
\end{equation}
Concretely, $f(m)(g)$ is the number of tuples of $O_F$-lattices $(Λ_r,\ldots,Λ_0)$ in $F^{2n}$ such that
\begin{equation}\label{eq:f_m}
gO_F^{2n} = Λ_0 \subseteq Λ_1 \subseteq \ldots \subseteq Λ_r = O_F^{2n},\qquad [Λ_i:Λ_{i-1}] = m_i\ \text{for all $i$}. 
\end{equation}
Also define $[π] = 1_{πGL_{2n}(O_F)}$ which is invertible with inverse $[π]^{-1} = 1_{π^{-1}GL_{2n}(O_F)}$. By Thm. \ref{mubiao},
\begin{equation}\label{eq:Hecke_algebra_generators}
\mcH = \sum_{k\in \mbZ,\ r\geq 0,\ m = (m_1,\ldots,m_r)\in \mbZ_{\geq 0}^r} \mbC\cdot [π]^k*f(m).
\end{equation}
We remark that this is not a direct sum, meaning there are non-trivial relations among the generators $[π]^k*f(m)$. We also mention that
\begin{equation}\label{eq:trivial_subst}
O(α, [π]^k*f, s) = O(α, f, s),\quad O(β, [π]^k*f) = O(β, f)
\end{equation}
for every $f\in \mcH$, $k\in \mbZ$. These identities are obtained from $([π]^k*f)(g) = f(π^{-k}g)$ and a variable substitution in the definitions of $O(α, f, s)$ and $O(β, f)$. Let $\mcH^j = \mbC[GL_{2n^j}(O_F)\backslash GL_{2n^j}(F) / GL_{2n^j}(O_F)]$, for $j = 0,1$, denote the spherical Hecke algebra of $GL_{2n^j}(F)$. There is a partial Satake transformation
$$\mcS: \mcH \lr \mcH^0\tensor_\mbC \mcH^1$$
which is discussed in detail in \eqref{eq:part_Satake} and Prop. \ref{explicit}. It is described on the above generators\footnote{In the appendix, the notation for $f(m)$ is $\mathbf T_m$ or $\mathbf T_{\GL_{2n}, m}$, see \eqref{eq:def_generators}.} by
\begin{equation}
\label{eq:Satake}
\mcS([π]^k*f(m)) = \sum_{m^0,\,m^1 \in \mbZ_{\geq 0}^r,\ m = m^0 + m^1} q^{n^1\cdot |m^0| + n^0\cdot |m^1|} ([π]^k*f(m^0)) \tensor ([π]^k*f(m^1)).
\end{equation}
Here, $|m| := \sum_{i = 1}^r m_i$ denotes the sum of all coefficients. Taking the product of the orbital integrals for $α^0$ and $α^1$, resp. for $β^0$ and $β^1$, defines orbital integrals
$$O^{\mr{Levi}}((α^0,α^1), f, s),\ \ \ O^{\mr{Levi}}((β^0,β^1), f),\ \ \ f\in \mcH^0\tensor_{\mbC}\mcH^1.$$
\begin{thm}\label{thm:analytic_reduction}
Given $α$ and $β$ as above, there are the following relations of orbital integrals on $GL_{2n}(F)$ and its Levi $GL_{2n^0}(F)\times GL_{2n^1}(F)$.
\begin{equation}\label{eq:analytic_reduction}
\begin{aligned}
O(α, f, s) & = &\left|\mathrm{Disc}_{E_3/F}\right|_F^{-\frac{n^0\cdot n^1}2} &
\left|\mr{Res}\left(\mathrm{Inv}(α^0), \mathrm{Inv}(α^1)\right)\right|_F^{-1}
\cdot O^{\mr{Levi}}((α^0,α^1), \mcS(f), s)\\
O(β, f) & = &\left|\mathrm{Disc}_{E_1/F}\cdot\mathrm{Disc}_{E_2/F}\right|_F^{-\frac{n^0\cdot n^1}2}&
\left|\mr{Res}\left(\mathrm{Inv}(β^0), \mathrm{Inv}(β^1)\right)\right|_F^{-1}
\cdot O^{\mr{Levi}}((β^0,β^1), \mcS(f)).
\end{aligned}
\end{equation}
Here, we have used the notation $|x|_F := |\mr{Nm}_{E_3/F}(x)|^{1/2}_F$ for elements $x\in E_3$.
\end{thm}

\begin{rmk}\label{rmk:res_non_zero}
The resultant $\mr{Res}(\Inv(α^0), \Inv(α^1))$ lies in $E_3^\times$. Namely, $α$ being regular semi-simple in particular means that $\Inv(α) = \Inv(α^0)\Inv(α^1)$ is separable, which implies that $\Inv(α^0)$ and $\Inv(α^1)$ have no common zeroes. The same applies to all other analogous resultant terms in this article.
\end{rmk}

\begin{proof}[Proof of Thm. \ref{thm:analytic_reduction}.]
These are \cite[Prop. 2.1 and 2.2]{Guo} when $E_1 = E_2$. We will give a proof for the general case in \S\ref{s:reduction}.
\end{proof}

\subsection{Vanishing order}
We now combine the reduction formula with the following vanishing statement.

\begin{prop}[\cite{HL}*{Prop. 3.3.3}]\label{prop:sign}
Assume that $α:(E_0,E_3)\to M_{2n}(F)$ is regular semi-simple and matches to a pair $β:(E_1, E_2)\to B$, where $B$ is a central division algebra of degree $2n$ over $F$. Then
$$O(α,f, 0) = 0.$$
\end{prop}
This relies on a functional equation for the orbital integral that relates $O(α, f, s)$ with $O(α, f, -s)$, we refer to \cite{HL}.

\begin{defn}\label{def:algebraic_order}
Let $δ$ be a regular semi-simple invariant and let $β_δ:(E_1,E_2)\to B_δ$ be the universal datum from Prop. \ref{prop:universal_pair}. Denote by $L_δ\subseteq B_δ$ the center, write $L_δ = \prod_{j\in J} L^j$ for its factorization into fields and let $B^j/L^j$ be the corresponding factor of $B_δ$. We define the order of $δ$ as
\begin{equation}
\label{eq:def_order}
\mr{ord}(δ) := |\{j\in J\mid \text{$B^j$ is a division algebra}\}|.
\end{equation}
\end{defn}

\begin{lem}\label{lem:embed}
Let $L/F$ be a field extension of degree $d$ and let $C/L$ be a quaternion algebra. There exists a central division algebra $D/F$ of degree $[D:F] = 2d$ together with an embedding $C\to D$ if and only if $C$ is non-split.
\end{lem}
\begin{proof}
Clearly, $C$ is non-split if there exists an embedding $C\to D$ as stated. The converse direction is obtained as follows. We denote the Hasse invariant of a central simple algebra $D/F$ by $\Ha(D)\in \mbQ/\mbZ$. Recall that $D\mapsto \Ha(D)$ defines an isomorphism $\mr{Br}(F)\simto \mbQ/\mbZ$. Also recall that for every field extension $L/F$, the identity $\Ha(L\tensor_FD) = [L:F] \cdot\Ha(D)$ holds. In particular, \cite{Draxl}*{Thm. 9.12 (c)} implies that for every central division algebra $D/F$ and for every field extension $L/F$ such that $[L:F]\mid [D:F]$, there exists an embedding $L\to D$ as $F$-algebras. In this situation, \cite{Draxl}*{Cor. 9.1} shows that $\Ha(\Cent_D(L)) = [L:F]\cdot \Ha(D)$. Thus if $[D:F] = 2[L:F]$, then there exists an embedding $L\to D$ and $\Cent_D(L)$ is a quaternion division algebra over $L$ which is isomorphic to $C$ if $C$ is non-split.
\end{proof}

\begin{cor}\label{cor:vanishing_order}
Let $α:(E_0,E_3)\to M_{2n}(F)$ be a regular semi-simple pair. Then
$$\mr{ord}_{s = 0} O(α, f, s) \geq \mr{ord}(\Inv(α)).$$
\end{cor}
\begin{proof}
Let $δ = \Inv(α)$ and let $β_δ:(E_1, E_2)\to B_δ$ be the universal datum for $(E_1, E_2, δ)$. Write $L_δ = \prod_{j\in J} L^j$ and $B_δ = \prod_{j\in J}B^j$ as before. Let $δ = \prod_{j\in J}δ^j$ be the corresponding factorization of $δ$, see \eqref{eq:descript_L}, and note that $B^j = B_{δ^j}$. By Lem. \ref{lem:embed}, $B^j$ embeds into a central division algebra of degree $2[L^j:F]$ if it is a division algebra. Then Prop. \ref{prop:sign} states that $O(α^j, f^j) = 0$ for every spherical Hecke function $f^j$ on $GL_{2[L^j:F]}(F)$ and every pair $α^j:(E_0, E_3)\to GL_{2[L^j:F]}(F)$ of invariant $δ^j$. Combining this observation with the product formula Thm. \ref{thm:analytic_reduction} and performing induction on $|J|$, we obtain the corollary.
\end{proof}

Let $q_F$ be the residue cardinality of $F$. We next consider the normalized first derivative
\begin{equation}\label{eq:def_delO}
\del(α,f) := \frac{1}{\log(q_F)} \left.\frac{d}{ds}\right\vert_{s = 0}O(α, f, s).
\end{equation}
\begin{cor}\label{cor:classification_for_derivative}
Assume that $α:(E_0, E_3) \to M_{2n}(F)$ is regular semi-simple and that $f\in \mcH$ is such that
\begin{equation}\label{eq:assume}
O(α,f,0) = 0,\ \ \ \del(α,f) \neq 0.
\end{equation}
Then $\mr{ord}(\mr{Inv}(α)) = 1$. In particular, there are uniquely determined $n^0, n^1$ with $n = n^0 + n^1$ such that $α$ matches a pair $β:(E_1,E_2)\to D_{1/2n^0} \times M_{2n^1}(F).$ Here, $D_λ$ denotes the central division algebra of Hasse invariant $λ$ over $F$.
\end{cor}
\begin{proof}
Set $\ord(α) = \ord(\Inv(α))$. Cor. \ref{cor:vanishing_order} provides $\mr{ord}(α) \leq 1$ and we are required to exclude the case $\mr{ord}(α) = 0$. This relies on the functional equation \cite[Prop. 3.2.8]{HL}, saying
\begin{equation}\label{eq:fctl_eq}
O(α,f,s) = ε(α,s)O(α,f,-s)
\end{equation}
for a factor $ε(α,s) \in \pm|F^\times|_F^s$ that only depends on $α$. The proof of \cite[Prop. 3.3.3]{HL}, combined with Thm. \ref{thm:analytic_reduction} and Lem. \ref{lem:embed}, shows that the sign of $ε(α,s)$ is $(-1)^{\mr{ord}(α)}$. Assume it were positive, say $ε(α,s) = q_F^{rs}$. Taking the derivative $(d/ds)\vert_{s=0}$ of \eqref{eq:fctl_eq}, we obtain that
$$\log(q_F) \del(α, f) = r\log(q_F)O(α, f) - \log(q_F) \del(α, f).$$
This is impossible under our assumptions \eqref{eq:assume}. Thus necessarily $\mr{ord}(α) = 1$. The existence of $(n^0, n^1)$ and of the pair $β$ is now evident from Lem. \ref{lem:embed}: Take $n^0 = [L^{j_0}:F]$ and $n^1 = n - n^0$, where $j_0\in J$ is the unique index in \eqref{eq:def_order} such that $B^{j_0}$ is division.
\end{proof}

\section{The Arithmetic Fundamental Lemma}
\label{s:AFL}
\subsection{Quadratic cycles}

Let $F$, $E_1$ and $E_2$ as well as $n\geq 1$ be as before. In particular, $E_1$ and $E_2$ are quadratic étale field extensions of $F$ of which at least one is unramified. Denote by $\breve F/F$ the completion of a maximal unramified extension and by $\mbF$ its residue field. Let $\breve E$ be a composite of $\breve F$, $E_1$ and $E_2$, and fix embeddings $E_1, E_2 \to \breve E$.

\begin{defn}\label{def:strict_O_F_module}
Let $S$ be an $O_{\breve F}$-scheme such that $π\in \mcO_S$ is locally nilpotent.
\begin{enumerate}[wide, labelindent=0pt, labelwidth=!, label=(\arabic*), topsep=2pt, itemsep=2pt]
\item Assume $F$ is $p$-adic. A strict $π$-divisible $O_F$-module over $S$ is a pair $(X, ι)$ consisting of a $p$-divisible group $X$ over $S$ and a strict action $ι:O_F\to \End(X)$. The latter means $ι(a) = a$ on $\Lie(X)$.
\item Assume $F \iso \mbF_q(\!(π)\!)$ is a local function field. A strict $π$-divisible $O_F$-module over $S$ is a $π$-divisible group\footnote{These are called $π$-divisible local Anderson modules in \cite{HS}. Our terminology is chosen for uniformity with the $p$-adic case.} $(X,ι)$ over $S$ (for $F$) in the sense of \cite[Def. 7.1]{HS} that is strict, meaning that $ι(a) = a$ on $\Lie(X)$.
\end{enumerate}
\end{defn}
The theories of these two types of objects are completely parallel, at least to the extent required by this article. We also remark that the objects in (2) are equivalent to minuscule local $O_F$-shtuka, cf. \cite[Thm. 8.3]{HS}.

Let $\mbX/\mbF$ be a strict $π$-divisible $O_F$-module of height $2n$ and dimension $1$ and consider a pair of embeddings
$$β:(E_1,E_2)\lr \End^0(\mbX).$$
Considering its identity component and maximal étale quotient, $\mbX$ decomposes in a unique way as a product $\mbX = \mbX^0 \times \mbX^1$
with $\mbX^0$ connected and $\mbX^1$ étale. The heights of $\mbX^0$ and $\mbX^1$ are even because of the existence of $β$ and we write $\mr{ht}(\mbX^0) = 2n^0$ and $\mr{ht}(\mbX^1) = 2n^1$. Then
$$\End(\mbX) \iso O_D \times M_{2n^1}(O_F)$$
where $D/F$ denotes a central division algebra of Hasse invariant $1/2n^0$. The datum of $β$ is thus equivalent to that of two pairs
$$β^0:(E_1,E_2) \to \End^0(\mbX^0),\ \ \ β^1:(E_1 , E_2) \to \End^0(\mbX^1).$$
\begin{lem}\label{lem:characterization_of_cases}
Assume that $β$ is regular semi-simple. Then $\mr{ord}(\Inv(β)) = 1$. Conversely, for every regular semi-simple invariant $δ$ with $\mr{ord}(δ) = 1$, there is a unique pair $(\mbX,β)$ as above up to isogeny such that $\Inv(β) = δ$.
\end{lem}
\begin{proof}
Write $L_δ = \prod L^j$ and $B_δ = \prod B^j$ as in Def. \ref{def:algebraic_order}. As seen during the proof of Cor. \ref{cor:classification_for_derivative}, there exists an $F$-algebra embedding $B_δ \to D_{1/2n^0} \times M_{2n^1}(F)$ if and only if there is a unique index $j_0$ with $B^{j_0}$ a division algebra and $[L^{j_0}:F] = n^0$. The pairs $(2n^0, 2n^1)$ with $n^0 + n^1 = n$ on the other hand are in bijection with the possible isogeny classes of $\mbX$ (Dieudonné--Manin classification). Finally, any embedding of $B_δ$ into $D_{1/2n^0}\times M_{2n^1}(F)$ has to be a product of two embeddings
$$B^{j_0}\lr D_{1/2n^0},\quad \prod_{j\neq j_0} B^j\lr M_{2n^1}(F)$$
and hence all such embeddings are conjugate by Skolem--Noether. This proves the uniqueness up to isogeny.
\end{proof}
\begin{rmk}\label{rmk:Drinfeld}
An analogous uniqueness statement holds for local Shimura data of EL type for inner forms of $GL_{2n}(F)$, see \cite{Li_Mih}*{Prop. 4.5}. It would be interesting to know if there is a more group-theoretic explanation for this phenomenon.
\end{rmk}

Let $M\to \Spf O_{\breve E}$ denote the RZ space of $\mbX$, that is
\begin{equation}
M(S) = \left\{(X, ρ)\ \left\vert\ \text{\begin{varwidth}{\textwidth}$X/S$ a strict $π$-divisible $O_F$-module\\
$ρ:\ov{S}\times_S X \to \ov{S}\times_{\Spec \mbF} \mbX$ a quasi-isogeny\end{varwidth}}\right\}\right..
\end{equation}
Here, $\ov{S} := \mbF\tensor_{O_{\breve E}} S$ denotes the special fiber of $S$. Analogously define, for $i = 1,2$,
\begin{equation}
Z(β_i)(S) = \left\{(X, ρ)\ \left\vert\ \text{\begin{varwidth}{\textwidth}$X/S$ a strict $π$-divisible $O_{E_i}$-module\\
$ρ:\ov{S}\times_S X \to \ov{S}\times_{\Spec \mbF} (\mbX, β_i)$ an $O_{E_i}$-linear quasi-isogeny\end{varwidth}}\right\}\right..
\end{equation}
The above spaces are well-known to be formally of finite type and formally smooth over $\Spf O_{\breve E}$. The relative dimension of $M$ is $2n-1$, that of the $Z(β_i)$ is $n-1$. Moreover, the reduced subschemes $M_{\mr{red}}$ and $Z(β_i)_{\mr{red}}$ are $0$-dimensional. The forgetful maps $Z(β_1), Z(β_2)\to M$ are closed immersions and allow to view the $Z(β_i)$ as cycles on $M$. 

\subsection{Hecke correspondences}

Throughout this article, we work with a non-standard definition of integral models of Hecke operators that is derived from our non-standard presentation \eqref{eq:Hecke_algebra_generators} of the Hecke algebra. Therefore, we begin with some general remarks on Hecke correspondences that justify this approach and provide the comparison with the definitions in \cite{Li} and  \cite{HL}.

Write $Z^c(M)$ for the group of cycles on $M$ of codimension $c$, similarly for $M\times_{\Spf O_{\breve E}} M$. By correspondence on $M$ we mean an element of
\begin{equation}
\mr{Corr}(M) := \left\{ z\in Z^{2n-1}(M\times_{\Spf O_{\breve E}} M) \left\vert \ p_1,p_2:\Supp(z) \to M\text{ both finite}\right\}\right..
\end{equation}
Correspondences form a ring through convolution and act on $Z^\bullet(M)$; we refer to Appendix \S\ref{s:appendix_loc_inter} for these definitions and their properties.

Recall that $\mcH$ denotes the spherical Hecke algebra of $GL_{2n}(F)$ with respect to $GL_{2n}(O_F)$. Given an indicator function $1_μ\in \mcH$, there is a definition of Hecke correspondence $[R(1_μ)] \in \mr{Corr}(M)$. It is uniquely characterized by the requirement that its generic fiber parametrizes pairs $\left((X_1,ρ_1), (X_2, ρ_2)\right)$ such that $ρ_2^{-1}ρ_1:X_1\to X_2$ is of type $μ$. The general definition of $[R(f)]$, $f\in \mcH$ is obtained by linear extension. Given cycles $z_1,z_2$ on $M$ of complementary dimension and such that $\Supp [R(f)] \cap (\Supp z_1 \times_{\Spf O_{\breve E}} \Supp z_2)$ is artinian, there is an intersection number $(z_1, [R(f)]*z_2)$. In this way, the Hecke action on cycles on $M$ and the resulting intersection numbers are unambiguously defined.

An explicit construction of $[R(f)]$ in terms of formal scheme models $R(f)\to M\times M$ is given in \cite[\S6.2]{Li} or \cite[\S4.3]{RSZ_unitary}, it relies on the notion of Drinfeld level structure. By Prop. \ref{prop:correspondences}, the Hecke action $[R(f)]:Z^\bullet(M) \to Z^\bullet(M)$ is then expressed as
$$[R(f)]*z = p_{1,*}[R(f)\times_{M\times M} p_2^{-1}(z)].$$
(The bracket notation indicates passage to the cycle.) Assume furthermore that $Z_1, Z_2\subset M$ are closed formal subschemes of pure complementary dimension, both defined by regular sequences, and such that $R(f) \cap (Z_1 \times_{\Spf O_{\breve E}} Z_2)$ is artinian. The mentioned models $R(f)$ are even regular, so Cor. \ref{cor:cycle_resolution_intersection_number} implies
\begin{equation}\label{eq:intersection_is_length}
\left([Z_1], [R(f)]*[Z_2]\right) = \ell_{O_{\breve E}} \big(\mcO_{R(f)\times_{M\times M} (Z_1 \times Z_2)}\big).
\end{equation}
This provides the promised link of the intersection numbers from \cite{Li, Li_future} (given by the RHS) with the canonical ones on the LHS. We now come to our definition of Hecke correspondences, which, by the above arguments, lead to the same intersection numbers.

\begin{defn}\label{def:Hecke_simple}
\begin{enumerate}[wide, labelindent=0pt, labelwidth=!, label=(\arabic*), topsep=2pt, itemsep=2pt]
\item Given $m\geq 0$, denote by $(p_1, p_2):R(m) \to M\times_{\Spf O_{\breve E}} M$ the functor
\begin{equation}
R(m)(S) = \left\{\left(X_0\overset{φ}{\to} X_1, ρ\right)\ \left\vert\ \text{\begin{varwidth}{\textwidth}$(X_1,ρ)\in M(S)$,\\
$φ$ an $O_F$-linear isogeny of height $m$\end{varwidth}}\right\}\right..
\end{equation}
The map to $M\times_{\Spf O_{\breve E}} M$ is given by
$$\left(X_0\overset{φ}{\to} X_1, ρ\right)\longmapsto \left((X_0, ρ\circ φ), (X_1, ρ)\right).$$
It is a closed immersion that identifies $R(m)$ with the locus of those pairs $((X_0, ρ_0), (X_1, ρ_1))$ such that the quasi-isogeny $ρ_1^{-1}ρ_0:X_0\to X_1$ is an isogeny of degree $m$. 
\item For a tuple $m = (m_1,\ldots,m_r)\in \mbZ_{\geq 0}$, we define
$$R(m) := R(m_1) \underset{p_2, M, p_1}{\times} R(m_2) \underset{p_2, M, p_1}{\times} \cdots \underset{p_2, M, p_1}{\times} R(m_r)$$
by composition. In other words, $R(m)$ represents the functor
\begin{equation}
R(m)(S) = \left\{\left(X_0\overset{φ_1}{\to} X_1\overset{φ_2}{\to} \ldots \overset{φ_r}{\to} X_r, ρ \right)\ \left\vert\ \text{\begin{varwidth}{\textwidth}Each $X_i$ a $π$-divisible strict $O_F$-module,\\
$(X_r,ρ)\in M(S)$,\\
$φ_i$ an $O_F$-linear isogeny of height $m_i$\end{varwidth}}\right\}\right..
\end{equation}
We often write $(X_\bullet, ρ)$ for such chains of isogenies. The map to $M\times_{\Spf O_{\breve E}} M$ takes the chain $(X_\bullet, ρ)$ to $((X_0, ρ\circ φ_r\circ \ldots \circ φ_1), (X_r, ρ))$.

\item For $k\in \mbZ$, multiplication in the framing defines an automorphism
$$[π]^k:M\longrightarrow M,\ \ (X,ρ)\longmapsto (X, π^kρ).$$
Taking its graph, we view $[π]^k:M\to M\times_{\Spf O_{\breve E}} M$ as correspondence.
\end{enumerate}
\end{defn}

\begin{prop}\label{prop:Hecke_op_flat}
For any $m = (m_1,\ldots,m_r)\in \mbZ_{\geq 0}$, the formal scheme $R(m)$ is Cohen--Macaulay and the two projection maps $p_1,p_2:R(m)\to M$ finite flat. In particular, Def. \ref{def:Hecke_simple} provides a model for the Hecke correspondence,
$$\left[[π]^k\underset{p_2, M, p_1}{\times} R(m)\right] = \left[R\left([π]^k*f(m)\right)\right]\ \ \text{in $\mr{Corr}(M)$}.$$
\end{prop}
\begin{proof}
First consider the case $r = 1$ and the first projection $p_1:R(m)\to M$. Let $(X, ρ)\in M(S)$. The fiber of $p_1$ over $(X, ρ)$ is the set of $O_F$-linear isogenies $X\to Y$ of height $m$ to strict $π$-divisible $O_F$-modules $Y$, up to isomorphism in $Y$. Considering the occurring kernels, this set is in bijection with certain finite locally free subgroups $H\subset X$ that we now describe.

In the $p$-adic setting, $H$ has to be $O_F$-stable, of order $q_F^m$, and such that $O_F$ acts on $X/H$ strictly, meaning $ι(a) = a$ in $\End(\Lie(X/H))$. Varying $S$ and $(X,ρ)$, this description shows that $p_1$ is relatively representable in projective schemes.

In the function field setting, we still consider $O_F$-stable $H$ of order $q_F^m$. We also need to make sure that $X/H$ is again a $π$-divisible $O_F$-module however. This is the case if and only if each truncation $(X/H)[π^r]$, $r\geq 0$, is balanced as $\mbF_{q_F}$-module scheme. (Being balanced is a condition on the primitive elements in the Hopf algebra of the group scheme, cf. \cite[Def. 5.13]{P}. This property is equivalent to being a strict $\mbF_q$-module scheme, cf. \cite[Rmk. 5.3]{HS}. It is required for the truncations of a $π$-divisible $O_F$-module by definition, cf. \cite[Def. 7.1]{HS}.) We are given that all $X[π^r]$ are balanced. Then all $(X/H)[π^r]$ are balanced if and only if $H$ is balanced, because this property is well-behaved in exact sequences. The condition for a finite locally free group scheme to be balanced is open and closed on the base, cf. \cite[Rmk. 5.17]{P}. Finally, we impose that the $O_F$-action on $X/H$ is strict. In this way, we have again shown that $p_1$ is relatively representable in projective schemes.

Now consider the special fiber $\Spec \mbF\times_M R(m)\to \Spec \mbF$. It is of dimension $0$ because a $1$-dimensional $π$-divisible $O_F$-module over $\mbF$ has at most finitely many subgroup of order $q_F^i$. Thus $p_1$ is finite and we obtain $\dim R(m) \leq \dim M$.

The map $R(m)\to M\times_{\Spf O_{\breve E}} M$ is a closed immersion and $M\times_{\Spf O_{\breve E}} M$ is regular. We claim that $R(m)$ is locally defined by at most $2n - 1$ equations. Let $X_i = p_i^*(X)$ be the pullback of the universal $π$-divisible $O_F$-module from the $i$-th component, $i = 1,2$. Write $R(m) = V(I)$ for an ideal sheaf $I$ and consider the Hodge filtrations of $X_1\vert_{V(I^2)}$ and $X_2\vert_{V(I^2)}$ as well as the map on ($O_F$-linear) Grothendieck--Messing crystals coming from $φ$:\footnote{We refer to \cite[Thm. 9.8]{HS} for the fact that the deformation theory of $π$-divisible $O_F$-modules in equal characteristic is the same as that in the $p$-adic setting.}
\begin{equation}\label{eq:diagram_Hecke_flat}
\xymatrix{
0 \ar[r]  & F_1 \ar[r] & D_{X_1}(V(I^2)) \ar[r] \ar[d]^{D_φ(V(I^2))}  & \Lie(X_1\vert_{V(I^2)}) \ar[r] & 0\\
0 \ar[r] & F_2 \ar[r] & D_{X_2}(V(I^2)) \ar[r] & \Lie(X_2\vert_{V(I^2)}) \ar[r] & 0.
}
\end{equation}
Note that $I/I^2 \subset \mcO_{V(I^2)}$ is defined by the condition that it is the minimal ideal $J\subset \mcO_{V(I^2)}$ such that the universal quasi-isogeny $ρ_2^{-1}ρ_1:X_1\vert_{V(I^2)}\to X_2\vert_{V(I^2)}$ is an isogeny mod $J$. (Here, $(X_1, ρ_1, X_2, ρ_2)$ denotes the universal point over $M\times_{\Spf O_{\breve E}} M$.) By Grothendieck--Messing theory, $I/I^2$ equals the ideal defined by the condition that the composite map $F_1\to \Lie(X_2\vert_{V(I^2)})$ in \eqref{eq:diagram_Hecke_flat} vanishes. The source is a vector bundle of rank $2n-1$, the target a line bundle. Thus $R(m)$ is locally defined by $2n - 1$ equations as claimed. In particular, also $\dim R(m) \geq \dim M$ everywhere, so $R(m)$ is pure of dimension $2n$. It follows from the smoothness of $M$ that $R(m)$ is Cohen--Macaulay.

The flatness of both $p_1$ and $p_2$ is now deduced from ``Miracle Flatness'': A local homomorphism $R\to S$ of noetherian local rings with $R$ regular and $S$ Cohen--Macaulay is flat if and only if its fiber dimension is constant, see \cite[00R4]{Stacks}.

The obtained properties now extend to general $r$: The composition of finite flat correspondences is finite flat, so both $p_1, p_2:R(m_1,\ldots,m_r)\to M$ are finite flat. Being finite flat over a regular local ring implies Cohen--Macaulay.

Finally, the identity $\left[[π]^k\times_{p_2,M,p_1} R(m)\right] = \left[R\left([π]^k*f(m)\right)\right]$ is clear for generic fibers. Since $R(m)$ is flat over $M$, it is also flat over $O_{\breve E}$, so the identity extends to cycles on $M$.
\end{proof}

\subsection{Intersection numbers}
\begin{defn}\label{def:intersection}
For $m = (m_1,\ldots, m_r)\in \mbZ_{\geq 0}$ and $β$ as before, we define $I(β,m)$ by the Cartesian diagram
\begin{equation}\label{eq:intersection_RZ}
\xymatrix{
I(β, m) \ar[r] \ar[d] \ar@{}[dr] | {\square} & Z(β_2)\times_{\Spf O_{\breve E}} Z(β_1) \ar[d]\\
R(m) \ar[r] & M\times_{\Spf O_{\breve E}} M.
}
\end{equation}
\end{defn}
\begin{lem}\label{lem:intersection_artinian}
Assume that $β$ is regular semi-simple. Then $I(β,m)$ is a (possibly infinite) union of artinian schemes.
\end{lem}
\begin{proof}
This follows a posteriori from the length computations in \cite{Li, HL}. We give an alternative a priori argument. Let $\Spf A \subset I(β,f)$ be an affine closed formal subscheme, given by some chain $(X_\bullet, ρ) \in R(m)(\Spf A)$. Then $X := X_r$ algebraizes to a strict $π$-divisible $O_F$-module $X$ over $\Spec A$ that comes with a pair of embeddings
\begin{equation}\label{eq:pair_geometric}
(ρ^{-1}β_1ρ,\ (ρφ_r\cdots φ_1)^{-1}β_2(ρφ_r\cdots φ_1)):(O_{E_1},O_{E_2}) \lr \End^0(X).
\end{equation}
The intersection of its image with $\End(X)$ generates an order in a quaternion algebra $B$ over an étale $F$-algebra $L$ of degree $n$ because $β$ is regular semi-simple, see Prop. \ref{prop:universal_pair}. Moreover $B\not\iso M_2(L)$ because $B$ embeds into $\End^0(\mbX)$. Our task is to show that $ρ$ algebraizes, i.e. that $(X_\bullet, ρ)\in R(m)(\Spec A)$.

It is well-known that $p$-divisible groups in characteristic $0$ are étale. The analogous statement holds in the function field setting. Namely let $X/S$ be a strict $π$-divisible $O_F$-module. The three properties $X$ being $π^\infty$-torsion, $π\in \mcO_S^\times$ and $ι(π) = π$ in $\End(\Lie(X))$ imply $\Lie(X) = 0$. The only strict $π$-divisible $O_F$-module of height $2n$ over a geometric point with $π\neq 0$ is thus $(F/O_F)^{2n}$ (up to isomorphism). Its endomorphism ring is $M_{2n}(O_F)$ and there is no possibility for a pair of invariant $\mr{Inv}(β)$ as in \eqref{eq:pair_geometric}. Thus $A$ has to be $π^k$-torsion, $k\gg 0$.

Concerning $O_F$-fields with $π = 0$, the geometric isogeny class of $\mbX$ is the only one of height $2n$ and dimension $1$ that admits a pair of embeddings of invariant $\mr{Inv}(β)$, see Lem. \ref{lem:characterization_of_cases}. Thus all fibers of $X$ are of the same geometric isogeny class. In particular, the connected component of identity $X^0$ is again a strict $π$-divisible $O_F$-module, cf. \cite[Prop. II.4.9]{Messing} ($p$-adic case) and \cite[Prop. 10.16]{HS} (function field case).

Let $A \to R$ be a local homomorphism to the ring of integers in a complete, algebraically closed, non-archimedean field $K$. Denote its residue field by $k$. We use index notation $X_R$, $X_K$, $X_k$ etc. to denote base changes to $R$, $K$, $k$ etc. The map
$$\End(\mbX_R) \lr \End(\mbX_K) = O_{D_{1/2n^0}} \times M_{2n^1}(O_F)$$
is injective, so the reduction map
$$\End(\mbX_R) \lr \End(\mbX_k)$$
is a bijection. If we manage to find an isogeny $τ:X_R \to \mbX_R$, then the composition $(ρτ^{-1}) τ:X_R \to \mbX_R$ is a $(\Spec R)$-valued point of $M$ that agrees with the given $\Spf R$-valued one. But any map $\Spec R \to M$ factors through one of the closed points of $M$ because $M$ is a disjoint union of formal spectra of complete local rings. Since $R$ was arbitrary, this implies the statement of the lemma.

Since $R$ is perfect, we may write $X_R = X_R^0 \times X_R^1$ and construct $τ$ component-wise
$$τ^0:X_R^0\lr \mbX_R^0,\ \ \ τ^1:X_R^1\lr \mbX_R^1.$$
Both $X_R^1$ and $\mbX_R^1$ are isomorphic to $(F/O_F)^{2n^1}$ and we take $τ^1$ as any isomorphism. For $τ^0$, we consider $X^0_R \iso \Spf R[\![t]\!]$ as a formal group. (See \cite[Cor. 10.12]{HS} for the equivalence of formal $O_F$-modules and connected strict $π$-divisible $O_F$-modules.) Since special and generic fiber of $X^0_R$ are both of height $2n^0$, and since $K$ is algebraically closed, we may normalize the coordinate such that multiplication by $π\in O_F$ is expressed by
$$[π](t) = t^{q^{2n^0}}.$$
The power series describing addition and the $O_F$-action all commute with $[π](t)$ and now have coefficients in the finite field $\mbF_{q^{2n^0}}$. In other words, $X^0_R$ comes via base change from $\mbF$ and is thus isomorphic to $\mbX^0_R$.
\end{proof}

Let $L := L(β) \subset \End^0(\mbX)$ denote the centralizer of $β_1(E_1)\cup β_2(E_2)$. Its units $L^\times$ act on all objects in \eqref{eq:intersection_RZ} by transport of structure along $ρ$.

\begin{lem}\label{lem:finiteness_intersection}
Assume that $β$ is regular semi-simple. Then the quotient $L^\times \backslash π_0(I(β,m))$ is finite.
\end{lem}
\begin{proof}
Since $I(β,m)(\mbF) = π_0(I(β,m))$, this is purely a statement about $\mbF$-points. Given a $π$-divisible group $X$ over $\mbF$, we write $X = X^0 \times X^1$ for its decomposition into connected and étale part. Then every $\mbF$-point $(X, ρ)\in M(\mbF)$ is (in a unique way) a product $(X^0 \times X^1, ρ^0\times ρ^1)$, where $ρ^0:X^0\to \mbX^0$ and $ρ^1:X^1\to \mbX^1$ are two quasi-isogenies. Moreover, giving an isogeny $φ_i:X_{i-1}\to X_i$ of height $m_i$ is the same as giving a pair of isogenies
$$φ_i^0:X_{i-1}^0\lr X_i^0,\ \ \ φ_i^1:X_{i-1}^1 \lr X_i^1$$
with $\mr{ht}(φ_i^0) + \mr{ht}(φ_i^1) = m_i$. In this way, we obtain
\begin{equation}\label{eq:connected_comp_intersection}
π_0(I(β, m)) = \coprod_{m^0 + m^1 = m} π_0(I(β^0, m^0)) \times π_0(I(β^1, m^1))
\end{equation}
where $m^0,m^1 \in \mbZ_{\geq 0}^r$. Moreover, the centralizer is a product $L = L^0 \times L^1$ and its units act diagonally on the right of \eqref{eq:connected_comp_intersection}. This reduces us to the two cases of $\mbX$ being connected or étale.

If $\mbX$ is connected, then $π_0(M) \iso \mbZ$ via $(X, ρ)\mapsto \mr{ht}(ρ)$. The element $π\in F^\times$ acts through translation by $2n$, proving the statement in this case. If $\mbX$ is étale, then $π_0(M) \iso GL_{2n}(F)/GL_{2n}(O_F)$ and the argument is the same as for the convergence of \eqref{def:orb_int_geom_side}.
\end{proof}

In fact, $[R(f)]$ is invariant under the diagonal action of $\Aut^0(\mbX)$ on $M\times_{\Spf O_{\breve E}} M$ for every $f$. Lem. \ref{lem:intersection_artinian} and Lem. \ref{lem:finiteness_intersection} imply that $\Supp [R(f)] \times_{M\times M} (Z(β_2) \times Z(β_1))$ is artinian and finite modulo $L^\times$ whenever $β$ is regular semi-simple. 

\begin{defn}\label{def:int_number_for_AFL}
Let $β$ be regular semi-simple. Endow $L = L(β)^\times$ with the measure such that $μ(O_L^\times) = 1$. Let $Γ\subseteq L^\times$ be a cocompact subgroup that acts freely on $π_0(M)$. For $f\in \mcH$, define the intersection number
\begin{equation}
\label{eq:int_number_general}
\begin{aligned}
\mr{Int}(β, f) :=\ & μ(Γ\backslash L^\times)^{-1} \left(Γ\backslash [R(f)],\ Γ\backslash (Z(β_2) \times Z(β_1))\right)\\[3mm]
=\ & \sum_{x \in L^\times\backslash π_0\left(\Supp [R(f)] \times (Z(β_2) \times Z(β_1))\right)} μ(\mr{Stab}(x))^{-1} \left([R(f)]_x,\ (Z(β_2)\times Z(β_1))_x\right).
\end{aligned}
\end{equation}
The intersection of the first expression is taken on $Γ\backslash (M\times M)$. The intersections of the second happen on the connected component $x$. More concretely, write $f = \sum_{i\in I} λ_i \cdot [π]^{k_i}* f(m^{(i)})$ for some finite index set $I$. (Recall that the $f(m^{(i)})$ were defined in \eqref{eq:def_Hecke_iterated}.) Then, by Prop. \ref{prop:Hecke_op_flat} and Cor. \ref{cor:cycle_resolution_intersection_number},
$$\mr{Int}(β,f) = \sum_{i\in I} λ_i \cdot \Int(β, m^{(i)})$$
where the summands on the right hand side are defined by
\begin{equation}
\label{eq:stacky_length}
\begin{aligned}
\mr{Int}(β, m) :=\ & μ(Γ\backslash L^\times)^{-1} \ell_{O_{\breve E}} (\mcO_{Γ\backslash I(β,m)})\\[3mm]
=\ & \sum_{x \in L^\times\backslash π_0(I(β,m))} μ(\mr{Stab}(x))^{-1} \ell_{O_{\breve E}}(\mcO_{I(β,m),x}).
\end{aligned}
\end{equation}
\end{defn}
Here, we have also used the following analog of \eqref{eq:trivial_subst}:
\begin{equation}\label{eq:trivial_inter}
\Int(β,f) = \Int(β, [π]^k*f),\quad f\in \mcH,\ k\in \mbZ.
\end{equation}
Namely, $[π]^k*R(f)$ is defined by composing $R(f)$ with the graph of the automorphism $π^{-k}:M\to M$, $(X, ρ)\mapsto (X, π^{-k}ρ)$, and this automorphism maps $Z(β_1)$ to itself. 

\subsection{The AFL}
Recall that $E_0 = F\times F$, and that $E_3 \subset E_1\tensor_F E_2$ is the fixed ring under $σ_1\tensor σ_2$. Recall further that
\begin{equation}
\del(α, f) = \frac{1}{\log(q_F)} \left.\frac{d}{ds}\right\vert_{s = 0}O(α, f, s)
\end{equation}
denotes the normalized first derivative. By Cor. \ref{cor:classification_for_derivative} and Lem. \ref{lem:characterization_of_cases}, the following conjecture concerns precisely those invariants $δ$ with $\mr{ord}(δ) = 1$.
\begin{conj}\label{conj:AFL}
Let $α:(E_0,E_3)\to M_{2n}(F)$ be a regular semi-simple pair of embeddings that matches to a pair $β:(E_1, E_2)\to \End^0(\mbX)$. Then, for any $f\in \mcH$, there is an equality
$$\del(α, f) = \Int(β, f).$$
\end{conj}

\begin{rmk}\label{rmk:AFL}
\begin{enumerate}[wide, labelindent=0pt, labelwidth=!, label=(\arabic*), topsep=2pt, itemsep=2pt]
\item In case $\mbX = \mbX^0$ connected, this is essentially the conjecture of Howard and the first author \cite{HL} (resp. the first author \cite{Li} if $E_1 = E_2$). The conjectures stated there are slightly different, however, because they do not take into account connected components. This is remedied with the definition in \eqref{eq:stacky_length}. The phenomenon does not yet occur for $n = 1$ or $n = 2$ because in these cases $L^\times$ acts transitively on $π_0(I(β, f))$. It is, unfortunately, quite challenging to verify an explicit example for $n = 3$ that shows that $\mr{Int}(β,f)$ is the correct definition. We offer a slightly different identity in Ex. \ref{ex:connected_components} below that provides some evidence however. We also remark that our definition of $\mr{Int}(β,f)$ is the one that comes naturally from a global intersection problem.

\item The case $n = 1$ of Conj. \ref{conj:AFL} is proved in \cite{HL, Li}. The case of $n = 2$ and $f = 1_{GL_4(O_F)}$ is shown in \cite{Li_GL4} (for $E_1 = E_2$) and the forthcoming work \cite{Li_future} (for $E_1\neq E_2$).

\item Assume $E_1 = E_2 = E$ is unramified over $F$. Given any pair $β$, one may write $β_2 = μβ_1μ^{-1}$ for some $μ\in D^\times$. Assume that $|m| \not\equiv v_D(γ)$ modulo $2$. Then $R(m) \cap (Z(β_2)\times Z(β_1)) = \emptyset$. One can also show (omitted) that the orbital integral then vanishes identically, $O(α,f,s) = 0$. In particular, Conj. \ref{conj:AFL} holds in that case. 
\end{enumerate}
\end{rmk}

It is enough to know Conj. \ref{conj:AFL} for a $\mbC$-basis of $\mcH$. Our main result to this end is the following analogue of the analytic reduction formula Thm. \ref{thm:analytic_reduction}.

\begin{thm}\label{thm:main}
For every $m \in \mbZ_{\geq 0}^r$, the following identity holds:
\begin{equation}
\begin{aligned}
\mr{Int}(β, m) = &\  \left|\mathrm{Disc}_{E_1/F}\cdot\mathrm{Disc}_{E_2/F} \right|_F^{-\frac{n^0\cdot n^1}2}
\left|\mr{Res}\left(\mathrm{Inv}(β^0), \mathrm{Inv}(β^1)\right)\right|_F^{-1}\\[3mm]
& \cdot \sum_{m = m^0 + m^1} q^{n^1\cdot |m^0| + n^0\cdot |m^1|}\, \mr{Int}(β^0, m^0) O(β^1, f(m^1)).
\end{aligned}
\end{equation}
Here, we have used the notation $|x|_F := |\mr{Nm}_{E_3/F}(x)|^{1/2}_F$ for elements $x\in E_3$.
\end{thm}
Comparing this with \eqref{eq:analytic_reduction}, we obtain the following corollary.

\begin{cor}\label{cor:main}
Assume that the FL (Conj. \ref{conj:FL}) holds. Then the AFL (Conj. \ref{conj:AFL}) holds in all cases if it holds in all basic cases, meaning in all cases where $\mbX = \mbX^0$ is connected.

More precisely, the AFL then holds for some $\mbX$, some $f(m)$ and some $β$ if it holds for $β^0$ and every $f(m^0)$ with $0\leq m^0\leq m$.

In particular, the AFL holds whenever $\mr{ht}(\mbX^0) = 2$. It holds for the unit function $f = 1_{GL_{2n}(O_F)}$ if $\mr{ht}(\mbX^0) = 4$.
\end{cor}
\begin{proof}
We show the more precise, second statement. Consider $\mbX = \mbX^0\times \mbX^1$ of height $2n = 2n^0+2n^1$, a pair $β = (β^0, β^1):(E_1, E_2)\to D_{2n^0}\times M_{2n^1}(F)$, and some $m \in \mbZ^r_{\geq 0}$. Let $α^0:(E_0, E_3)\to M_{2n^0}(F)$ match $β^0$ and let $α^1:(E_0, E_3)\to M_{2n^1}(F)$ match $β^1$. In particular, their direct sum
$$α = (α^0, α^1):(E_0, E_3) \to M_{2n}(F)$$
matches $β$. Set $δ^j := \Inv(α^j) = \Inv(β^j)$ in the following. Also define
$$\Theta = \left|\mathrm{Disc}_{E_1/F}\cdot\mathrm{Disc}_{E_2/F} \right|_F^{-\frac{n^0\cdot n^1}2} \left|\mr{Res}\left(δ^0, δ^1\right)\right|_F^{-1}$$
and note that this expression equals
$$\left|\mathrm{Disc}_{E_3/F}\right|_F^{-\frac{n^0\cdot n^1}2} \left|\mr{Res}\left(δ^0, δ^1\right)\right|_F^{-1}.$$
Namely, taking $E_1$ to be unramified, $\mr{Disc}_{E_1/F}$ is trivial and
$$\mr{Disc}_{E_2/F} = \mr{Disc}_{E_1\tensor_F E_2/E_1} = \mr{Disc}_{E_1\tensor_F E_3/E_1} = \mr{Disc}_{E_3/F}$$
where the outer identities follow from the unramifiedness of $E_1/F$ and where the middle one follows from $E_1\tensor_FE_3\iso E_1\tensor_FE_2$. Now we apply Thm. \ref{thm:analytic_reduction}. Substituting the concrete expression for the partial Satake transformation $\mcS(f(m))$ from \eqref{eq:Satake}, taking the derivative at $s = 0$, and applying the vanishing of $O(α^0,f(m^0))$ for every $m^0\in \mbZ^r_{\geq 0}$ (see Prop. \ref{prop:sign}), this theorem states that
\begin{equation}\label{eq:1}
\del(α, f(m)) = \Theta \cdot \sum_{m = m^0 + m^1} q^{n^1\cdot |m^0| + n^0\cdot |m^1|}\, \del(α^0, f(m^0)) O(α^1, f(m^1)).
\end{equation}
On the other hand, Thm. \ref{thm:main} states that
\begin{equation}\label{eq:2}
\mr{Int}(β, m) = \Theta \cdot \sum_{m = m^0 + m^1} q^{n^1\cdot |m^0| + n^0\cdot |m^1|}\, \mr{Int}(β^0, m^0) O(β^1, f(m^1)).
\end{equation}
The fundamental lemma (Conj. \ref{conj:FL}), which we assumed to hold, states that $O(α^1, f(m^1))$ equals $O(β^1, f(m^1))$ for all $m^1$. Furthermore, our assumption also is that $\del(α^0, f(m^0)) = \mr{Int}(β^0, m^0)$ for all occurring $m^0$. Comparing \eqref{eq:1} and \eqref{eq:2} then shows that $\del(α, f(m)) = \mr{Int}(β, m)$ as was to be shown. The last statement of the corollary follows from the known cases of the AFL in Rmk. \ref{rmk:AFL}.
\end{proof}

The proof of Thm. \ref{thm:main} will be given in the next section. Here, we conclude with the promised example on connected component counts.

\begin{ex}\label{ex:connected_components}
Assume that $\mbX = \mbX^0$ is connected and that $E = E_1 = E_2$ is unramified over $F$. Let $L/F$ be an unramified field extension of degree $n$ and choose an embedding $L\to D = \End^0(\mbX)$. Write $B = \mr{Cent}_D(L)$ for its centralizer and let $β:(E, E)\to B$ be a pair of embeddings of $F$-algebras whose image generates $B$ over $F$. This situation can only exist if $n$ is odd because $E\tensor_FL$ then embeds into $B$ by Prop. \ref{prop:universal_pair} (2) and hence has to be a field. Denote by $Z = Z(O_L, β)\subseteq M$ the closed formal subscheme of $(X,ρ)$ such that
$$ρ^{-1} O_Lρ,\ \ ρ^{-1}β_1(O_E)ρ,\ \ ρ^{-1}β_2(O_E)ρ\ \ \mr{all}\ \subseteq \End(X).$$
Let $Γ = π^\mbZ$. (Note that $π$ is also a uniformizer for $L$.) Then $Γ\backslash Z(O_L, β)$ has $2n$ connected components that correspond to the $2n$ possible characters of $O_{β_1(E)L}$, acting on the Lie algebra. Imposing compatible Kottwitz conditions for the two $O_E$-actions leaves us with only $n$ connected components. Each is isomorphic to the intersection $Γ\backslash I(β \tensor \mr{id}_L, 1_{GL_2(O_L)})$ that occurs in the formulation of the linear AFL for $GL_{2,L}$. So
$$\ell_{O_{\breve E}}(\mcO_{\Gamma\back Z}) = n \Int(β \tensor \mr{id}_L, 1_{GL_2(O_L)}).$$
Let $α:(E_0,E_3)\to M_{2n}(F)$ match $β$. The centralizer of $α$ is isomorphic to $L$. In \eqref{eq:lattice_count_2}, we may sum only over the sets
$$\mcL(O_L[α_0])\subseteq \mcL(α_0),\ \ \ \mcL(O_L[α_3])\subseteq \mcL(α_3)$$ of lattices $Λ$ that are also $O_L$-stable to obtain a rational function in $q^s$,
\begin{equation}\label{eq:lattice_count_L_variant}
O(O_L[α],1_{GL_{2n}(O_F)},s) = \sum_{(Λ_0, Λ_3) \in L^\times \backslash (\mathcal{L}(O_L[α_0])\times \mathcal{L}(O_L[α_3]))} \mr{Vol}(\mr{Stab}_{L^\times}(Λ_0,Λ_3))^{-1} δ(Λ_0,Λ_3) \Omega(Λ_0, Λ_3, s),
\end{equation}
with $δ(Λ_0,Λ_3) = 1$ if $Λ_0 = Λ_3$ and $0$ otherwise. Then
\begin{equation}\label{eq:orb_int_identities_L}
\begin{aligned}
O(O_L[α], 1_{GL_{2n}(O_F)}, s) & = O(α \tensor \mr{id}_L, 1_{GL_2(O_L)}, s),\\[3mm]
\del(O_L[α], 1_{GL_{2n}(O_F)}) & = n \del(α \tensor \mr{id}_L, 1_{GL_2(O_L)}),
\end{aligned}
\end{equation}
where the factor $n$ stems from the difference in normalization by $q_F$ or $q_L = q_F^n$. In this way, we have recovered the AFL conjecture over $L$ for $n = 2$ and the unit Hecke function, but only with the right convention for connected component counting.
\end{ex}

\section{Reduction to elliptic invariant}
\label{s:reduction}
The purpose of this section is to prove the reduction formulas Thm. \ref{thm:analytic_reduction} and Thm. \ref{thm:main}. The next three sections reduce this to a degree computation that will be accomplished in \S\ref{ss:zhongxin}.

\subsection{Fibration of lattice sets}
Let $V = V^0\oplus V^1$ be a direct sum decomposition of $V = F^{2n}$ with $V^0 \iso F^{2n^0}$ and $V^1 \iso F^{2n^1}$. Denote by $p^j:V\to V^j$ the two projection maps. The fundamental observation is that there is a bijection of $O_F$-lattices $X\subseteq V$ and triples
$$\left\{\left(X^0, X^1, s\right)\ \left\vert\ \text{\begin{varwidth}{\textwidth}$X^0\subset V^0,\ X^1\subset V^1$ both $O_F$-lattices,\\
$s:X^1 \to V^0/X^0$ any $O_F$-linear map\end{varwidth}}\right\}\right..$$
It is given by
\begin{equation}
X \longmapsto \left(X^0 := X\cap V^0,\ X^1 := p^1(X),\ s_X := [X^1\overset{t}{\longrightarrow}X\overset{p^0}{\longrightarrow} V^0/X^0]\right)
\end{equation}
where $t:X^1\to X$ is any choice of splitting for $p^1$. The inverse is given by
$$\left(X^0, X^1, s\right) \longmapsto X^0 + \left\{(\wt s(x_1),x_1),\ x_1\in X^1\right\}$$
where $\wt s:X^1 \to V^0$ is any lift of $s$. The bijection moreover satisfies the following two functoriality properties.
\begin{enumerate}[wide, labelindent=0pt, labelwidth=!, label=(\arabic*), topsep=2pt, itemsep=2pt]
\item Assume $X\subseteq Y$. Then clearly $X^0 \subseteq Y^0$ and $X^1 \subseteq Y^1$. We claim that also the following diagram commutes
\begin{equation}\label{eq:comm_square_lattice_fibration}
\xymatrix{X^1 \ar[d]_{s_X} \ar[r] & Y^1 \ar[d]^{s_Y}\\
V^0/X^0 \ar[r] & V^0/Y^0.}
\end{equation}
Namely $\left(\wt s_X(x^1), x^1\right) \in X$ lying in $Y$ means that there is $y^0\in Y^0$ such that $\wt s_X(x^1) = \wt s_Y(x^1) + y_0.$ Conversely, if for two lattices $X$ and $Y$ one has $X^j \subseteq Y^j$ and also that \eqref{eq:comm_square_lattice_fibration} commutes, then $X\subseteq Y$.

\item Let $ζ = (ζ^0, ζ^1) \in \End(V^0)\times \End(V^1)\subset \End(V)$ be an endomorphism. Then $X$ is $ζ$-stable if and only if each $X^j$ is $ζ^j$-stable and $ζ^0 \circ s = s \circ ζ^1$.
\end{enumerate}

\subsection{Application to orbital integrals}
Consider a regular semi-simple pair $α:(E_0,E_3)\to M_{2n}(F)$. Let $B$ be the quaternion algebra generated by the image of $α$, denote by $L = \mr{Cent}(B)$ its center. Assume that $L = L^0\times L^1$ is a product with eigenspace decomposition $V = V^0\oplus V^1$. Given $m = (m_1,\ldots,m_r)\in \mbZ_{\geq 0}^r$, define
\begin{equation}\label{eq:def_lattice_set_E03}
\mathcal{L}(α,m) = \left\{(X_0, X_1, \ldots,X_r)\ \left\vert\ \text{\begin{varwidth}{\textwidth}$X_0\subset V$ an $O_{E_3}$-lattice, $X_r\subset V$ an $O_{E_0}$-lattice,\\
$X_i\subset V$ an $O_F$-lattice for $i =1,\ldots,r-1$,\\
$X_{i-1}\subseteq X_i$ of index $m_i$\end{varwidth}}\right\}\right..
\end{equation}
Note that this definition matches \eqref{eq:f_m}. The above discussion provides a map
\begin{equation}\label{eq:fibration_lattices_E03}
\begin{aligned}
\mathcal{L}(α,m) \ & \longrightarrow\ \coprod_{m^0,m^1\in \mbZ_{\geq 0}^r,\ m = m^0+m^1} \mathcal{L}(α^0,m^0) \times \mathcal{L}(α^1,m^1)\\[3mm]
(X_0,\ldots,X_r) & \longmapsto \ \ \ \ \ \ \ \ \ \left((X_0^0,\ldots,X_r^0), (X_0^1,\ldots,X_r^1)\right).
\end{aligned}
\end{equation}
We use the shorter notation $X_\bullet$, $X_\bullet^0$ and $X_\bullet^1$ in the following. Let $Γ^j\subset L^{j, \times}$ be a subgroup with $Γ^j \times \mcO_{L^j}^\times \iso L^{j,\times}$ and put $Γ = Γ^0\times Γ^1 \subseteq L^\times$. Then \eqref{eq:fibration_lattices_E03} is $Γ$-equivariant. Using \eqref{eq:lattice_count_2} and the description of $f(m)$ in \eqref{eq:f_m}, we obtain
\begin{equation}
O\left(α,[π]^k*f(m), s\right) = \sum_{X_\bullet \in Γ\backslash \mathcal{L}(α,m)} \Omega(X_r, X_0, s).
\end{equation}
The same results give
\begin{multline}
\quad O\left(α^0, [π]^k*f(m^0), s\right) O\left(α^1, [π]^k*f(m^1), s\right)\\[3mm]
= \sum_{(X_\bullet^0, X_\bullet^1)\in Γ^0\backslash \mathcal{L}(α^0,m^0)\times Γ^1\backslash \mathcal{L}(α^1, m^1)} \Omega(X_r^0, X_0^0, s) \Omega(X_r^1, X_0^1, s)\qquad
\end{multline}
Observe now that the transfer factor from \eqref{eq:root_number} is multiplicative: If $X_\bullet \mapsto (X_\bullet^0, X_\bullet^1)$ in \eqref{eq:fibration_lattices_E03}, then
\begin{equation}\label{eq:product_Omega}
\Omega(X_r, X_0, s) = \Omega(X_r^0, X_0^0, s)\,\Omega(X_r^1, X_0^1, s).
\end{equation}
By the explicit description of the partial Satake transform \eqref{eq:Satake}, the reduction formula Thm. \ref{thm:analytic_reduction} (for $E_0$ and $E_3$) is now reduced to the following proposition.

\begin{prop}\label{prop:fiber_count_analytic_E03}
Each fiber of \eqref{eq:fibration_lattices_E03} over $\mathcal{L}(α^0,m^0)\times \mathcal{L}(α^1,m^1)$ has
\begin{equation}\label{eq:reduction_factor_E03}
\left|\mathrm{Disc}_{E_3/F}\right|_F^{-\frac{n^0\cdot n^1}2}\left|\Res\left(\mathrm{Inv}(α^0), \mathrm{Inv}(α^1)\right)\right|_F^{-1}\cdot q^{n^1|m^0|+n^0\cdot |m^1|}
\end{equation}
many elements. In other words, given any pair $(X_\bullet^0,X_\bullet^1) \in \mathcal{L}(α^0,m^0)\times \mathcal{L}(α^1,m^1)$, there are \eqref{eq:reduction_factor_E03} many tuples $(γ_0,\ldots,γ_r)$ of $O_F$-linear maps $γ_i:X^1_i \to V^0/X^0_i$ that fit into a commutative diagram
\begin{equation}\label{eq:diagram_conditions_E03}
\xymatrix{
X_0^1 \ar[r]\ar[d]_{γ_0}& X_1^1 \ar[r]\ar[d]_{γ_1}& \cdots \ar[r]& X_r^1\ar[d]_{γ_r}\\
V^0/X_0^0 \ar[r]& V^0/X_1^0 \ar[r]& \cdots \ar[r]& V^0/X_r^0\\
}
\end{equation}
and such that $γ_0$ is $O_{E_3}$-linear, and such that $γ_r$ is $O_{E_0}$-linear.
\end{prop}

Given a regular semi-simple pair $β:(E_1, E_2)\to M_{2n}(F)$ there is a completely analogous definition of a lattice set $\mathcal{L}(β, m)$. If $β$ factors through $M_{2n^0}(F)\times M_{2n^1}(F)$, all the above considerations apply to provide a fibration
\begin{equation}\label{eq:fibration_lattices_E12}
\mathcal{L}(β, m)\longrightarrow \coprod_{m = m^0 + m^1} \mathcal{L}(β^0, m^0)\times \mathcal{L}(β^1, m^1).
\end{equation}
In this way, Thm. \ref{thm:analytic_reduction} (for $E_1$ and $E_2$) reduces to the following statement.
\begin{prop}\label{prop:fiber_count_analytic_E12}
Each fiber of \eqref{eq:fibration_lattices_E12} over $\mathcal{L}(β^0,m^0)\times \mathcal{L}(β^1,m^1)$ has
\begin{equation}\label{eq:reduction_factor_E12}
\left|\mathrm{Disc}_{E_1/F}\cdot \mathrm{Disc}_{E_2/F}\right|_F^{-\frac{n^0\cdot n^1}2}\left|\Res\left(\mathrm{Inv}(β^0), \mathrm{Inv}(β^1)\right)\right|_F^{-1}\cdot q^{n^1|m^0|+n^0\cdot |m^1|}
\end{equation}
many elements. In other words, given any pair $(X_\bullet^0,X_\bullet^1) \in \mathcal{L}(β^0,m^0)\times \mathcal{L}(β^1,m^1)$, there are \eqref{eq:reduction_factor_E12} many tuples $(γ_0,\ldots,γ_r)$ of $O_F$-linear maps $γ_i:X^1_i \to V^0/X^0_i$ that fit into a commutative diagram
\begin{equation}\label{eq:diagram_conditions_E12}
\xymatrix{
X_0^1 \ar[r]\ar[d]_{γ_0}& X_1^1 \ar[r]\ar[d]_{γ_1}& \cdots \ar[r]& X_r^1\ar[d]_{γ_r}\\
V^0/X_0^0 \ar[r]& V^0/X_1^0 \ar[r]& \cdots \ar[r]& V^0/X_r^0\\
}
\end{equation}
and such that $γ_0$ is $O_{E_2}$-linear, and such that $γ_r$ is $O_{E_1}$-linear.
\end{prop}

\subsection{Fibration of RZ spaces}
\label{ss:fibration_RZ}
Our aim is to adapt the previous arguments to strict $π$-divisible $O_F$-modules. We begin with some general observations about the connected-étale sequence. Let $S$ be a scheme with $π\in \mcO_S$ locally nilpotent. Assume that $X/S$ has fiberwise constant étale rank. Then, by \cite[Prop. II.4.9]{Messing} resp. \cite[Prop. 10.16]{HS}, it is an extension
$$0 \longrightarrow X^0 \longrightarrow X \longrightarrow X^1 \longrightarrow 0$$
of its maximal étale quotient by its identity component. We view $π$-divisible $O_F$-modules as sheaves for the fpqc-topology in the following. Then the maximal étale quotient $X^1$ has the canonical presentation
\begin{equation}\label{eq:tate_module_seq}
0 \longrightarrow T \longrightarrow V \longrightarrow X^1 \longrightarrow 0
\end{equation}
where
$$T := TX^1 := \underset{\longleftarrow}{\lim}{\,}_{n\geq 0}\,\, X^1[π^n]\quad \text{and}\quad V := VX^1 := TX^1[π^{-1}]$$
denote the Tate module resp. rational Tate module of $X^1$. The natural map $V\to X^1$ in \eqref{eq:tate_module_seq} is defined by
$$π^{-n}(\ldots,x_2,x_1,x_0) \longmapsto x_n,\ \ \text{where}\ (\ldots,x_2,x_1,x_0)\in TX^1.$$
Applying $\Hom(-, X^0)$ and taking into account that multiplication by $π$ is an automorphism of $V$ while $\Hom(T,X^0)$ is torsion, we obtain an exact sequence
\begin{equation}\label{eq:hom_ext_sequence}
0 \longrightarrow \Hom(T, X^0) \longrightarrow \Ext^1(X^1,X^0) \longrightarrow \Ext^1(V, X^0).
\end{equation}
The leftmost term here is torsion, the rightmost term torsion-free, so
$$\Hom(T, X^0) = \Ext^1(X^1, X^0)_{\mr{torsion}}.$$
By definition, an extension
$$0 \longrightarrow X^0 \longrightarrow X \longrightarrow X^1 \longrightarrow 0$$
here is torsion if and only if there exists an integer $N\geq 1$ and a commutative diagram
\begin{equation}\label{eq:extension_torsion}
\xymatrix{
0 \ar[r] & X^0 \ar[r] \ar@{=}[d] & X^0 \oplus X^1 \ar[r] \ar[d] & X^1 \ar[r] \ar[d]^-{π^N} & 0\\
0 \ar[r]		& X^0 \ar[r]			& X \ar[r]				& X^1 \ar[r] & 0.}
\end{equation}
\begin{prop}\label{prop:extension_p_div}
\begin{enumerate}[wide, labelindent=0pt, labelwidth=!, label=(\arabic*), topsep=2pt, itemsep=2pt]
\item An extension $0\to X^0\to X \to X^1\to 0$ is torsion if and only if there exists an isogeny $φ:X\to X^0\oplus X^1$.

\item Assume $φ^0:X^0 \to Y^0$ and $φ^1:X^1\to Y^1$ are two homomorphisms with $Y^1$ étale. Assume furthermore that we are given two torsion extensions
$$0\lr X^0\lr X \lr X^1\lr 0\ \ \ \mr{and}\ \ \ 0\lr Y^0\lr Y \lr Y^1\lr 0.$$
Then there exists a (necessarily unique) map $φ:X \to Y$ that restricts to $φ_0$ and extends $φ_1$ if and only if the following diagram commutes
$$\xymatrix{
TX^1 \ar[r]^{Tφ^1} \ar[d]_X & TY^1 \ar[d]^Y \\
X^0 \ar[r]^{φ^0} & Y^0.
}$$
\end{enumerate}
\end{prop}
\begin{proof}
(1) The two $π$-divisible $O_F$-modules in the middle column of \eqref{eq:extension_torsion} have the same height and the kernel of the middle vertical arrow of \eqref{eq:extension_torsion} identifies with $X^1[π^N]$, is hence finite locally free. This shows the ``only if'' assertion.

For the converse, we observe that any homomorphism $φ:X\to X^0\oplus X^1$ provides two homomorphisms $φ^0 = φ\vert_{X^0}:X^0\to X^0$ and $φ^1 = (φ\mod X^0): X^1\to X^1$. If $\ker(φ)$ is finite locally free, then also the intersection $\ker(φ^0) = \ker(φ) \cap X^0$ has this property because it is open and closed in $\ker(φ)$. Thus $φ^0$ is an isogeny. Then $\ker(φ^1) = \ker(φ)/\ker(φ^0)$ is also finite locally free, so $φ^1$ is an isogeny. Let $ψ^0$ and $ψ^1$ be such that $φ^1ψ^1 = π^N$ and $ψ^0φ^0 = π^N$ for some $N\geq 0$. Then the class of $X$ in $\Ext^1$ is annihilated by $π^{2N}$.

(2) The connecting homomorphism in \eqref{eq:hom_ext_sequence} is such that $X$ is constructed from a map $s_X:TX^1\to X^0$ as $X = (X^0 \oplus VX^1)/TX^1$. The claim is then obvious.
\end{proof}
Apply these considerations now to $(X,ρ)$, the universal pair over $M$. Let
\begin{equation}\label{eq:conn_et_universal}
0\lr X^0\lr X\lr X^1 \lr 0
\end{equation}
be its connected-étale sequence and let $S\to M$ be a point. Both $X^0$ and $X^1$ inherit a strict $O_F$-action and $ρ$ provides a pair of quasi-isogenies
\begin{equation}\label{eq:compat_framing}
ρ^0:\ov S \times_S X^0 \lr \ov S\times_{\Spec \mbF} \mbX^0,\ \ \ ρ^1:\ov S \times_S X^1 \lr \ov S\times_{\Spec \mbF} \mbX^1
\end{equation}
over the special fiber $\ov S = \mbF\tensor_{O_{\breve E}} S$. It follows that $\ov S \times_{S} X$ is isogeneous to $\ov S \times_S (X^0\times X^1)$ and hence $X$ is isogeneous to $X^0\times X^1$. The extension class $S\times_M X$ is thus torsion by Prop. \ref{prop:extension_p_div} (1) and comes from a map $TX^1\to X^0$.

Let $M_{\mbX^0}$ and $M_{\mbX^1}$ denote the RZ spaces for the two indicated strict $π$-divisible $O_F$-modules. Denote by $(Y^0, ρ^0)$ and $(Y^1, ρ^1)$ the universal points of $M_{\mbX^0}$ resp. $M_{\mbX^1}$. By abuse of notation, we continue to write $(Y^j, ρ^j)$ in place of $p_j^*(Y^j, ρ^j)$ for their pullbacks to $M_{\mbX^0}\times M_{\mbX^1}$. The Hom-functor $\underline{\Hom}(TY^1, Y^0)$ is then itself a strict $π$-divisible $O_F$-module. (In fact, since $M_{\mbX^0}\times M_{\mbX^1}$ is a union of formal spectra of strict complete local rings, it is isomorphic to $(Y^0)^{2n^1}$.) Thus we obtain a map
\begin{equation}\label{eq:fibration_RZ_space}
\begin{aligned}
M_\mbX & \longrightarrow \underline{\Hom}(TY^1, Y^0)\\
(X,ρ) & \longmapsto \big((X^0,ρ^0),\ (X^1, ρ^1),\ \text{extension class of $X$}\big).
\end{aligned}
\end{equation}
Here, we view $\underline{\Hom}(TY^1, Y^0)$ as a sheaf on $M_{\mbX^0}\times M_{\mbX^1}$ and $(X^j, ρ^j)$ denotes the image point in $M_{\mbX^j}$. Prop. \ref{prop:extension_p_div} proves the following proposition.
\begin{prop}\label{prop:fibration_RZ_space}
The map \eqref{eq:fibration_RZ_space} is an isomorphism. It moreover has the following functoriality property. Given $ζ = (ζ^0, ζ^1)\in \End^0(\mbX^0)\times \End^0(\mbX^1)$, the closed formal subscheme $\mcZ(ζ) \subseteq M_\mbX$ where $ρ^{-1}ζρ$ is a homomorphism agrees with the subfunctor of $ζ$-linear homomorphisms in
\begin{equation}\label{eq:x_lin_prep}
\left(\mcZ(ζ^0)\times \mcZ(ζ^1)\right)\times_{(M_{\mbX^0}\times M_{\mbX^1})} \underline{\Hom}(TY^1, Y^0).
\end{equation}
Here, an $S$-valued point $(X^0, ρ^0, X^1, ρ^1, h:TX^0\to X^1)$ is called $ζ$-linear if
$$(ρ^1)^{-1}ζ^1ρ^1\circ h = h\circ T((ρ^0)^{-1}ζ^0ρ^0).$$
\end{prop}

Consider now a regular semi-simple pair $β:(E_1,E_2)\to \End^0(\mbX)$, write $β^j:(E_1,E_2)\to \End^0(\mbX^j)$, $j = 0,1$ for its components. Then we have the three intersections $I(β,m)$, $I(β^0,m^0)$ and $I(β^1,m^1)$ that result from \eqref{eq:intersection_RZ} applied to $\mbX$, $\mbX^0$ and $\mbX^1$, respectively. The above considerations provide a map
\begin{equation}\label{eq:fibration_intersection}
\begin{aligned}
\Pi:I(β,m) & \longrightarrow \coprod_{m = m^0 + m^1} I(β^0,m^0)\times I(β^1,m^1)\\
(X_0\overset{φ_1}{\to} X_1\overset{φ_2}{\to} \ldots \overset{φ_r}{\to} X_r, ρ ) &
\longmapsto
\left((X_0^0\overset{φ_1}{\to} \ldots \overset{φ_r}{\to} X_r^0, ρ^0),\ (X_0^1\overset{φ_1}{\to} \ldots \overset{φ_r}{\to} X_r^1, ρ^1)\right).
\end{aligned}
\end{equation}
\begin{prop}\label{prop:fiber_count_geometric}
The map \eqref{eq:fibration_intersection} is finite locally free. Its degree over $I(β^0,m^0)\times I(β^1,m^1)$ equals
\begin{equation}\label{eq:reduction_factor_geometric}
\left|\mathrm{Disc}_{E_1/F}\cdot\mathrm{Disc}_{E_2/F}\right|_F^{-\frac{n^0\cdot n^1}2}\left|\Res\left(\mathrm{Inv}(β^0), \mathrm{Inv}(β^1)\right)\right|_F^{-1}\cdot q^{n^1\cdot |m^0|+n^0\cdot |m^1|}.
\end{equation}
\end{prop}
The proof will be given in the next section. Here we note the following description of $I(β, m)$, which is provided by Prop. \ref{prop:extension_p_div} and Prop. \ref{prop:fibration_RZ_space}. Namely
$$I(β,m) \subset \prod_{i = 0}^r \underline{\Hom}_{O_F}(TY_i^1, Y_i^0)\vert_{\coprod_{m = m^0 + m^1} I(β^0, m^0)\times I(β^1, m^1)}$$
is the subfunctor of those tuples $(γ_i:TY_i^1\to Y_i^0)_{i = 0}^r$ such that the diagram
\begin{equation}
\label{eq:diagram_condition_intersection}
\xymatrix{
TY_0^1 \ar[r]^{Tφ_1^1} \ar[d]_{γ_0}& TY_1^1 \ar[r]^{Tφ_2^1}\ar[d]_{γ_1}& \cdots \ar[r]^{Tφ_r^1}& TY_r^1\ar[d]_{γ_r}\\
Y_0^0 \ar[r]^{φ_1^0}& Y_1^0 \ar[r]^{φ_2^0}& \cdots \ar[r]^{φ_r^0} & Y_r^0\\
}
\end{equation}
commutes, such that $γ_0$ is $O_{E_2}$-linear, and such that $γ_r$ is $O_{E_1}$-linear.

\subsection{Degree computation}
\label{ss:zhongxin}

This section proves Prop. \ref{prop:fiber_count_analytic_E03}, Prop. \ref{prop:fiber_count_analytic_E12} and Prop. \ref{prop:fiber_count_geometric} in parallel, the arguments being the same for all three cases. Our main interest lies in the RZ space intersection, so we use the terminology of $π$-divisible $O_F$-modules. Fix a pair $β:(E_1,E_2)\to \End^0(\mbX)$, a tuple $m\in \mbZ_{\geq 0}^r$ and a decomposition $m = m^0 + m^1$. Let $Y_\bullet^1 \in I(β^1, m^1)(S)$ and $Y_\bullet^0\in I(β^0,m^0)(S)$ be two fixed $S$-valued points where $S$ is local artinian with residue field $\mbF$. Our aim is to compute the fiber over $(Y_\bullet^0, Y_\bullet^1)$ of
\begin{equation}\label{eq:fibration_RZ_sect_4}
I(β,m)\longrightarrow \coprod_{m = m^0 + m^1} I(β^0,m^0)\times I(β^1,m^1).
\end{equation}
The group $Y_\bullet^1$ only comes in via its Tate module, so we introduce the notation $T_\bullet := TY_\bullet^1$ and put $Y_\bullet = Y_\bullet^0.$

For the proofs in the analytic settings, one considers instead a pair $α:(E_0,E_3)\to M_{2n}(F)$ (resp. $β:(E_1,E_2)\to M_{2n}(F)$) that preserve a decomposition $V = V^0\oplus V^1$. Fix lattice chains $Λ_\bullet \in \mathcal{L}(α^0,m^0)$ and $T_\bullet \in \mathcal{L}(α^1,m^1)$, (resp. $Λ_\bullet \in \mathcal{L}(β^0,m^0)$ and $T_\bullet \in \mathcal{L}(β^1,m^1)$). Taking $S = \Spec \mbF$, the datum $Λ_\bullet$ is equivalent to that of a chain $Y_\bullet := Λ_\bullet \tensor_{O_F} (F/O_F)$ of étale $π$-divisible $O_F$-modules of type $(α^0,m^0)$, (resp. $(β^0, m^0)$). Similarly, $T_\bullet$ is the chain of Tate modules of $T_\bullet \tensor_{O_F} (F/O_F)$. In this way, all arguments about $π$-divisible $O_F$-modules in the following apply literally and provide the proofs of Prop. \ref{prop:fiber_count_analytic_E03} and Prop. \ref{prop:fiber_count_analytic_E12}. We only formulate the case of $(E_1, E_2)$ however, leaving it to the reader to substitute for $(E_0, E_3)$.

Consider the following $π$-divisible $O_F$-modules over $S$,
\begin{equation}\label{def:O_F_modules_I}
\begin{aligned}
H_i & = \Hom_{O_F}(T_i, Y_i),\ i = 0,\ldots,r\\
C_i & = \Hom_{O_F}(T_{i-1}, Y_i),\ i = 1,\ldots,r
\end{aligned}
\end{equation}
as well as
\begin{equation}\label{def:O_F_modules_II}
\begin{aligned}
H_0^+ & = \Hom_{O_{E_2}}(T_0, Y_0),\ \ \ \ \ H_0^-  = \Hom_{O_{E_2}\text{-conj}}(T_0, Y_0),\\
H_r^+ & = \Hom_{O_{E_1}}(T_r, Y_r),\ \ \ \ \ H_r^-  = \Hom_{O_{E_1}\text{-conj}}(T_r, Y_r).
\end{aligned}
\end{equation}
Here, we used the notation
$$\Hom_{O_{E_2}\text{-conj}}(T_0, Y_0) = \{f:T_0\to Y_0 \mid f(at) = a^{σ_2}f(t)\ \forall\ a\in O_{E_2}\}$$
in the first line and an analogous notation in the second. We will now define a homomorphism
\begin{equation}\label{eq:def_phi}
Φ = \m{p_2^-}{}{}{},
{-R_1}{L_1}{}{},
{}\ddots\ddots{},
{}{}{-R_r}{L_r},
{}{}{}{p_1^-}.: \bigoplus_{i = 0}^r H_i \longrightarrow H_0^- \oplus \bigoplus_{i = 1}^r C_i \oplus H_r^-
\end{equation}
such that $\ker(Φ)$ precisely describes the tuples $(γ_0, \ldots, γ_r)$ in \eqref{eq:diagram_conditions_E03}, \eqref{eq:diagram_conditions_E12} and \eqref{eq:diagram_condition_intersection}. The components $R_i$ and $L_i$ are the composition to the right and left in each square,
\begin{equation}\label{eq:def_LR}
\begin{aligned}
R_i:H_{i-1} & \longrightarrow C_i,\quad γ\longmapsto [Y_{i-1} \to Y_i]\circ γ,\ \ \ i = 1, \ldots, r\\
L_i:H_i & \longrightarrow C_i,\quad γ\longmapsto γ\circ [T_{i-1}\to T_i],\ \ \ i = 1,\ldots, r.
\end{aligned}
\end{equation}
The condition $L_i(γ_i) - R_i(γ_{i-1}) = 0$ precisely expresses commutativity of the square
$$\xymatrix{
T_{i-1} \ar[r] \ar[d]_{γ_{i-1}} & T_i \ar[d]^{γ_i}\\
Y_{i-1} \ar[r] & Y_i.}
$$
The kernels of $p_2^-:H_0\to H_0^-$ and $p_1^-:H_r\to H_r^-$ on the other hand should precisely be the $O_{E_2}$-linear (resp. $O_{E_1}$-linear) elements. Pick elements $ζ_i$ such that $O_{E_i} = O_F[ζ_i]$. Define
\begin{equation}\label{eq:def_p_minus}
\begin{aligned}
p_2^-:H_0 & \longrightarrow H^-_0,\ f \longmapsto [t\mapsto f(ζ_2t) - ζ_2f(t)]\\
p_1^-:H_r & \longrightarrow H^-_r,\ f \longmapsto [t\mapsto f(ζ_1t) - ζ_1f(t)].
\end{aligned}
\end{equation}
It is easily checked that $p_2^-$ and $p_1^-$ have image in $H^-_0$ resp. $H^-_r$ as claimed. For example,
$$\begin{aligned}
p_1^-(f)(ζ_1t) & = f(ζ_1^2t) - ζ_1f(ζ_1t)\\
& = (ζ_1 + ζ_1^{σ_1}) f(ζ_1t) - ζ_1ζ_1^{σ_1} f(t) - ζ_1 f(ζ_1t)\\
& = ζ_1^{σ_1} [f(ζ_1t) - ζ_1f(t)]\\
& = ζ_1^{σ_1}p_1^-(f)(t).
\end{aligned}$$
This concludes the definition of \eqref{eq:def_phi}.

\begin{lem}\label{lem:p_div_property}
Let $E/F$ be an étale quadratic extension and $T$ an $O_E$-lattice. Let $Y$ be a $π$-divisible group with $O_E$-action, write $O_E = O_F[ζ]$. Then $\Hom_{O_E}(T, Y)$ is again a $π$-divisible $O_F$-module and equals the kernel of
$$p:\Hom_{O_F}(T, Y) \lr \Hom_{O_E\text{-conj}}(T, Y),\ f \longmapsto [t \mapsto f(ζt) - ζf(t)].$$
In particular, the following sequence is exact,
\begin{equation}\label{eq:ex_seq_projection}
0 \longrightarrow \Hom_{O_E}(T, Y)\overset{ι}{\lr} \Hom_{O_F}(T, Y) \overset{p}{\longrightarrow} \Hom_{O_E\text{-conj}}(T, Y) \longrightarrow 0.
\end{equation}
\end{lem}
\begin{proof}
Choosing a basis of $T$, it suffices to consider $T = O_E$. Then both $\Hom_{O_E}(O_E, Y)$ and $\Hom_{O_E\text{-conj}}(O_E, Y)$ are isomorphic to $Y$ by the map $f\mapsto f(1)$. The choice of $ζ$ also allows to identify
$$\Hom_{O_F}(O_E, Y) \iso Y^2,\ f \longmapsto (f(1), f(ζ)).$$
In these coordinates, $p$ is given by $p(y_1, y_2) = y_2 - ζy_1$. Its kernel is isomorphic to $Y$ via $(y_1, ζy_1)\mapsto y_1$ and hence a $π$-divisible $O_F$-module as claimed. Then the cokernel of $ι$ is a $π$-divisible $O_F$-module and has the same height as $\Hom_{O_E\text{-conj}}(T, Y)$, so maps isomorphically onto it.
\end{proof}

In particular, $\ker(Φ)$ precisely describes the fiber of \eqref{eq:fibration_RZ_sect_4}. The maps $R_i$ and $L_i$ are isogenies, so we have quasi-isogenies $L_i^{-1}R_i:H_{i-1} \dashrightarrow H_i$ that allow to define a quasi-homomorphism
$$\wt{p}_1^- := p_1^- L_r^{-1}R_rL_{r-1}^{-1}R_{r-1}\cdots L_1^{-1}R_1: H_0 \dashrightarrow H_r^-.$$
\begin{lem}\label{lem:phi_isogeny}
The homomorphism $Φ$ is an isogeny if and only if the quasi-homomorphism
$$\m{p_2^-}, {\wt{p}_1^-}.: H_0\dashrightarrow H_0^-\oplus H_r^-$$
is a quasi-isogeny. If both are quasi-isogenies, then their degrees are related by
\begin{equation}\label{eq:deg_relation}
\deg(Φ) = \deg\m{p_2^-}, {\wt {p}_1^-}. \cdot\prod_{i=1}^r\deg(L_i).
\end{equation}
\end{lem}
\begin{proof}
We may multiply $Φ$ from the left by a quasi-isogeny of degree $1$ in the following way,
$$
\m
{I_{\frac n2}}{}{}{}{},
{}{I_{n}}{}{}{},
{}{}\ddots{}{},
{}{}{}{-p_1^-L_r^{-1}}{I_{\frac n2}},
{}{}{}{I_{n}}{}.
\m{p_2^-}{}{}{},
{-R_1}{L_1}{}{},
{}\ddots\ddots{},
{}{}{-R_r}{L_r},
{}{}{}{p_1^-}.
=
\m{p_2^-}{}{}{}{},
{-R_1}{L_1}{}{}{},
{}\ddots\ddots{}{},
{}{}{R_{r-1}}{L_{r-1}}{},
{}{}{}{p_1^-L_r^{-1}R_r}{},
{}{}{}{-R_r}{L_r}..
$$
Iterating this operation leaves us with
$$
\m{p_1^-}{}{}{},
{\wt{p}_2^-}{}{}{},
{-R_1}{L_1}{}{},
{}\ddots\ddots{},
{}{}{-R_r}{L_r}..
$$
The degree of the diagonal here is the right hand side of \eqref{eq:deg_relation} and the claim follows.
\end{proof}
We now shift our attention to the framing object. The link with Lem. \ref{lem:phi_isogeny} will be given at the end of the section. Write
$$N = N(\mbX),\ \ N^0 = N(\mbX^0),\ \ N^1 = N(\mbX^1)$$
for the isocrystals of the respective $O_F$-modules. These are $\breve F$-vector spaces of dimensions $2n$, $2n^0$ and $2n^1$, respectively. Also put
$$P = N(\Hom_{O_F}(T\mbX^1, \mbX^0)) = \Hom(N^1,N^0)$$
which has $\breve F$-dimension $4n^0n^1$. Define the following subspaces of $P$, each of half that dimension,
\begin{equation}\label{eq:subspaces_P}
\begin{aligned}
P^+_1 = \Hom_{E_1}(N^1,N^0),& \ \ \ P^-_1 = \Hom_{E_1\text{-conj}}(N^1, N^0)\\
P^+_2 = \Hom_{E_2}(N^1,N^0),& \ \ \ P^-_2 = \Hom_{E_2\text{-conj}}(N^1, N^0).
\end{aligned}
\end{equation}
We also consider the projection maps to $P^{\pm}_i$, for $i = 1,2$,
\begin{equation}\label{eq:def_q}
\begin{aligned}
q_i^+:P & \longrightarrow P^+_i,\ \ f \longmapsto [t\mapsto f(ζ_it) - ζ_i^{σ_i}f(t)]\\
q_i^-:P & \longrightarrow P^-_i,\ \ f \longmapsto [t\mapsto f(ζ_it) - ζ_i^{σ_i}f(t)].
\end{aligned}
\end{equation}
The given points $(Y_\bullet^j, ρ_\bullet^j, φ_\bullet^j) \in I(β^j,m^j)(S)$ define lattices in the above $\breve F$-vector spaces. Let $M(-)$ denotes the integral Dieudonné module functor over the special point of $S$. Set
\begin{equation}\label{eq:def_lattices}
\begin{aligned}
M^0_1 = N(ρ_r^0)(M(Y_r^0)),&\qquad M^0_2 = N(ρ_0^0)(M(Y_0^0))\\[1mm]
M^1_1 = N(ρ_r^1)(M(Y_r^1)),&\qquad M^1_2 = N(ρ_0^1)(M(Y_0^1)).
\end{aligned}
\end{equation}
Then $M^0_1 \subset N^0$ and $M^1_1\subset N^1$ are $O_{E_1}$-stable, while $M^0_2$ and $M^1_2$ are $O_{E_2}$-stable. We note for later use that there is an inclusion $M^0_2\subseteq M^0_1$ of index $|m^0|$ because of the chain of isogenies
$$Y_0^0 \overset{φ^0_1}{\lr} Y_1^0 \overset{φ^0_2}{\lr} \ldots \overset{φ^0_r}{\lr} Y_r^0$$
that satsfies $ρ_0^0 = ρ_r^0\circ φ^0_r \circ \cdots \circ φ^0_1$. In the same way, $M^1_2\subseteq M^1_1$ with index $|m^1|$.

Passing to $\Hom$-spaces, these lattices provide Dieudonné lattices in $P$ and the $P_i^{\pm}$,
\begin{equation}
\begin{aligned}
Λ_i & = \Hom_{O_F}(M_i^1, M_i^0)\\
Λ_i^+ & = \Hom_{O_{E_i}}(M_i^1, M_i^0)\\
Λ_i^- & = \Hom_{O_{E_i}\text{-conj}}(M_i^1, M_i^0).
\end{aligned}
\end{equation}
Just like in \eqref{eq:ex_seq_projection}, the projections from \eqref{eq:def_q} restrict to surjections $q_i^{\pm}:Λ_i\twoheadrightarrow Λ_i^{\pm}$. Let us adopt the following notion: Given two $O_{\breve F}$-lattices or $O_F$-lattices $Λ$ and $Λ'$, we write $f:Λ\dashrightarrow Λ'$ for homomorphisms $f:Λ[π^{-1}]\to Λ'[π^{-1}]$ together with the datum of the two lattices. We call such morphisms quasi-homomorphisms. We call a quasi-homomorphism $f$ a quasi-isogeny if $f:Λ[π^{-1}]\to Λ'[π^{-1}]$ is bijective. In this case, the degree of $f$ is defined as the valuation of the ratio $\det(f(Λ'))/\det(Λ)$. For example, $Λ_1$ and $Λ_2$ are both lattices in the $\breve F$-vector space $P$. The inclusion maps define a quasi-isogeny $Λ_2 \dashrightarrow Λ_1$ with
\begin{equation}\label{eq:degree_lambdas}
\begin{aligned}
\deg(Λ_2\dashrightarrow Λ_1) &\ =\ [Λ_1:Λ_2]\\
&\ = \ [M^0_1:M^0_2] \dim_{\breve F}(N^1) - [M^1_1:M^1_2] \dim_{\breve F}(N^0)\\
&\ = \ 2n^1|m^0| - 2n^0|m^1|.
\end{aligned}
\end{equation}

\begin{prop}\label{prop:key_degree_computation}
Assume that $β$ is regular semi-simple. Then the quasi-homomorphism
$$(q_1^-, q_2^-):Λ_2 \dashrightarrow Λ_1^-\times Λ_2^-$$
is a quasi-isogeny and its degree is
\begin{equation}\label{eq:yaoling}
\left|\mathrm{Disc}_{E_1/F}\cdot\mathrm{Disc}_{E_2/F}\right|_F^{-\frac{n^0\cdot n^1}2}\left|\Res\left(\mathrm{Inv}(β^0), \mathrm{Inv}(β^1)\right)\right|_F^{-1}\cdot q^{n^1\cdot|m^0| - n^0 \cdot|m^1|}.
\end{equation}
Here, we have used the notation $|x|_F := |\mr{Nm}_{E_3/F}(x)|^{1/2}_F$ for elements $x\in E_3$.
\end{prop}

We first prove two lemmas.

\begin{lem}\label{lem:degree_identities}
If one of the six maps appearing below is a quasi-isogeny, then all of them are. In this case, the degrees of
\begin{equation}\label{eq:three_1}
\begin{array}{c}
(q_1^-, q_2^-): Λ_1 \dashrightarrow Λ_1^-\times Λ_2^-,\\[3mm]
q_2^-\vert_{Λ_1^+}:Λ_1^+ \dashrightarrow Λ_2^-\ \ \mr{and}\ \ q_2^+\vert_{Λ_1^-}:Λ_1^-\dashrightarrow Λ_2^+
\end{array}
\end{equation}
all agree. Similarly, the degrees of
\begin{equation}\label{eq:three_2}
\begin{array}{c}
(q_1^-, q_2^-):Λ_2 \dashrightarrow Λ_1^-\times Λ_2^-,\\[3mm]
q_1^-\vert_{Λ_2^+}:Λ_2^+ \dashrightarrow Λ_1^-\ \ \mr{and}\ \ q_1^+\vert_{Λ_2^-}:Λ_2^-\dashrightarrow Λ_1^+
\end{array}
\end{equation}
all agree.
\end{lem}
\begin{proof}
We prove the claim for the first triple of maps. Consider the exact sequence from \eqref{eq:ex_seq_projection} to obtain the following commutative diagram
$$
\xymatrix{
0\ar[r]& Λ_1^+ \ar@{-->}[d]_{q_{2}^-} \ar[r] & Λ_1 \ar[r]^{q_{1}^-}\ar@{-->}[d]_{\m{q_1^-},{q_{2}^-}.}& Λ_1^- \ar[r]\ar@{=}[d] & 0\\
0\ar[r] & Λ_2^- \ar[r] & Λ_1^-\times Λ_2^- \ar[r]^-{\mr{pr}} & Λ_1^- \ar[r] & 0.
}
$$
(Note that $q_1^-:Λ_1\to Λ_1^-$ is surjective even if $E_1/F$ is ramified; the argument is as for the map $p$ in \eqref{eq:ex_seq_projection}.) The claim follows from this diagram for the two maps $(q_1^-, q_2^-)$ and $q_2^-\vert_{Λ_1^+}$ in \eqref{eq:three_1}. Consider now any quasi-isogeny $z\in \End^0_F(\mbX^0)$ that Galois commutes with both $E_1$ and $E_2$ in the sense that $zζ_i = ζ_i^{σ_i}z$ for both $i = 1,2$. (Such elements always exist; a specific choice is given by the element $\mathbf{z}$ from \cite[(2.4.2)]{HL}. It Galois commutes with both $E_i$ by \cite[Prop. 2.4.2]{HL} and is invertible because $β^0$ is regular semi-simple. The latter implies $\mathbf{t}$ in \emph{loc. cit.} to be invertible by definition, then \cite[Prop. 2.4.1]{HL} implies $\mathbf{z}$ invertible.) The element $z$ intertwines $q_2^-\vert_{Λ_1^+}$ and $q_2^+\vert_{Λ_1^-}$ in the sense that the following diagram commutes.
$$\xymatrix{
Λ_1^+
\ar@{-->}[rr]^{f\longmapsto z \circ f}
\ar@{-->}[d]_{q_2^-}
&&
Λ_1^-
\ar@{-->}[d]^{q_2^+}
\\
Λ_2^-\ar@{-->}[rr]^{f\longmapsto z\circ f} && Λ_2^+.
}$$
Since $z$ is a quasi-isogeny and since the horizontal maps have the same degree, this proves the claimed shared properties of the first triple. The case of the second triple is the same by symmetry. Moreover, $Λ_1$ and $Λ_2$ are lattices in the same isocrystal, so the first three maps are quasi-isogenies if and only if the second triple is. 
\end{proof}

\begin{lem}\label{lem:resultant_occurs}
The endomorphism $q_1^+q_2^-\vert_{Λ_1^+}:Λ_1^+ \to Λ_1^+$ is a quasi-isogeny of degree
$$
\left|\mathrm{Disc}_{E_1/F}\cdot\mathrm{Disc}_{E_2/F}\right|_F^{-n^0\cdot n^1}\left|\Res\left(\mathrm{Inv}(β^0), \mathrm{Inv}(β^1)\right)\right|_F^{-2}.
$$
\end{lem}
\begin{proof}
For any $f\in Λ_1^+$, we have $β_1^0(x)\circ f = f\circ β_1^1(x)$ for every $x\in E_1$. From definition \eqref{eq:def_q} we obtain that
\begin{equation}\label{eq:q_composition}
\[split]{
q_1^+(q_2^-(f))
= &
\underbrace{\left(
	f\circ β_2^1 (ζ_2)
	- 
	β_2^0(ζ_2)\circ f 
\right)
}_{\text{definition of $q_2^-(f)$ in \eqref{eq:def_q}}}
\circ β_1^1(ζ_1)\\
 & -
β_1^0(ζ^{σ_1}_1) \circ
\underbrace{
\left(
	β_2^0(ζ^{σ_2}_2)\circ f 
	- 
	f\circ β_2^1(ζ^{σ_2}_2)
\right)
}_{\text{equals $q_2^-(f)$, see \eqref{eq:simple_identity} below}}\\
= &
f\circ w(β^1)
-
w(β^0)\circ f
}
\end{equation}
with the two elements
\begin{equation}\label{eq:def_w}
\begin{aligned}
w(β^0) & = \left(β_2^0(ζ_2) \circ β_1^0(ζ_1) + β_1^0(ζ^{σ_1}_1)\circ β_2^0(ζ^{σ_2}_2)\right)\\
w(β^1) & = \left(β_2^1(ζ_2) \circ β_1^1(ζ_1) + β_1^1(ζ^{σ_1}_1)\circ β_2^1(ζ^{σ_2}_2)\right).
\end{aligned}
\end{equation}
We have also used the simple observation
\begin{equation}\label{eq:simple_identity}
f\circ ζ_i - ζ_i \circ f = ζ_i^{σ_i}\circ f - f \circ ζ_i^{σ_i}
\end{equation}
for all maps $f$, which results from $f \circ \mr{tr}(ζ_i) = \mr{tr}(ζ_i) \circ f.$
The elements $w(β^0)$ and $w(β^1)$ are precisely the elements $\mathbf w$ for the two pairs $β^0$ and $β^1$ from \cite[(2.4.1)]{HL}. They have the property that each of them is both $O_{E_1}$- and $O_{E_2}$-linear, i.e. $w(β^i)$ centralizes $β^i$, see \cite[Prop. 2.4.2]{HL}. Note that $\End_{E_1}^0(\mbX^1) \iso M_{n^1}(E_1)$, also write $D = \End_{E_1}^0(\mbX^0)$. Thus we may view
$$w(β^0) \in D,\ \ w(β^1) \in M_{n^1}(E_1)$$
and $q_1^+ q_2^-\vert_{Λ_1^+} = w(β^1)\tensor 1 - 1 \tensor w(β^0) \in M_{n^1}(E_1) \tensor_{E_1} D$.
We recall a basic fact about the resultant. Assume we are given $A\in M_n(k)$ and $B\in M_m(k)$ where $k$ is any field. Then
$$\det\big[A\tensor 1 - 1 \tensor B: M_{n\times m}(k)\longrightarrow M_{n\times m}(k),\ f\longmapsto Af - fB\big] = \mr{Res}(P_A, P_B),$$
where $P_A$ and $P_B$ denote the characteristic polynomials of $A$ and $B$. (This identity is clear if $A$ and $B$ are diagonalizable. The general case follows from a Zariski density argument and by passing to the algebraic closure of $k$.) Put differently, the determinant of $A\tensor 1 - 1 \tensor B \in M_n(k)\tensor_k M_m(k)$, viewed as $(nm\times nm)$-matrix is $\mr{Res}(P_A, P_B)$. Since determinants may be computed after scalar extension, these considerations also apply to central simple algebras and we obtain
\begin{equation}\label{eq:deg_comp_intermediate}
\begin{aligned}
\deg(q_1^+ q_2^-\vert_{Λ_1^+}) &= \left\vert\mr{Nrd}_{M_{n^1}(D)/E_1}\left(w(β^1)\tensor 1 - 1 \tensor w(β^0)\right)\right\vert^{-1}_{E_1}\\
&= \left\vert\mr{Res}(\mr{charred}(w(β^0)), \mr{charred}(w(β^1)))\right\vert^{-1}_{E_1}\\
&= \left\vert\mr{Res}(\mr{charred}(w(β^0)), \mr{charred}(w(β^1)))^2\right\vert^{-1}_{F}.
\end{aligned}
\end{equation}
It is shown in \cite[Prop. 2.4.1]{HL} that the two characteristic polynomials here are related to the invariants of $β^0$ and $β^1$ by a linear transformation. More precisely, one has $w(β^i) = d\mathbf{s}_{β^i} + c$ in $E_3 \tensor_F M_{2n^1}(F)$ resp. $E_3\tensor_F D_{1/2n^0}$ where $c$, $d\in E_3$ are the following two elements,
$$c = -(ζ_1\tensor ζ_2^{σ_2} + ζ_1^{σ_1}\tensor ζ_2),\ \ \ d = (ζ_1-ζ_1^{σ_1})\tensor (ζ_2 - ζ_2^{σ_2}).$$
(These a priori lie in $E_1\tensor_FE_2$ but are clearly $σ_1\tensor σ_2$-invariant.) It follows that if $μ_1,\ldots,μ_{n^i}\in \bar F$ are the eigenvalues of $\Inv(β^i)$ with respect to a homomorphism $ρ:E_3\to \bar F$, then $dμ_1+c,\ldots,dμ_{n^i}+c$ are the eigenvalues of $\charred(w(β^i))$ with respect to $ρ$. The resultant term in \eqref{eq:deg_comp_intermediate} is thus equal to
\begin{equation}
\left\vert\left((ζ_1 - ζ_1^{σ_1})(ζ_2- ζ_2^{σ_2})\right)^{2n^0\cdot n^1}\mr{Res}(\mr{Inv}(β^0), \mr{Inv}(β^1))^2\right\vert_F^{-1}.
\end{equation}
Since $|ζ_i - ζ_i^{σ_i}|_F = |\mr{Disc}_{E_i/F}|_F^{1/2}$, the lemma is proved.
\end{proof}

\begin{proof}[Proof of Prop. \ref{prop:key_degree_computation}.]
Lem. \ref{lem:degree_identities} and Lem. \ref{lem:resultant_occurs} already show that $(q_1^-,q_2^-):Λ_2\dashrightarrow Λ_1^-\times Λ_2^-$ is a quasi-isogeny. From Lem. \ref{lem:degree_identities}, we obtain
\begin{equation}\label{eq:key_trick}
\begin{aligned}
\deg\big((q_1^-&, q_2^-): Λ_2 \dashrightarrow Λ_1^-\times Λ_2^-\big){}^2\\[1mm]
& = \deg(Λ_2\dashrightarrow Λ_1) \deg\big((q_1^-, q_2^-): Λ_1 \dashrightarrow Λ_1^-\times Λ_2^-\big) \deg \big((q_1^-, q_2^-): Λ_2 \dashrightarrow Λ_1^-\times Λ_2^-\big)\\[1mm]
& = \deg(Λ_1\dashrightarrow Λ_2) \deg(q_2^-:Λ_1^+\dashrightarrow Λ_2^-)\deg(q_1^+:Λ_2^-\dashrightarrow Λ_1^+)\\[1mm]
& = \deg(Λ_1\dashrightarrow Λ_2) \deg(q_1^+q_2^-: Λ_1^+\dashrightarrow Λ_1^+).
\end{aligned}
\end{equation}
Substituting \eqref{eq:degree_lambdas} and Lem. \ref{lem:resultant_occurs}, this means that $\deg\big((q_1^-, q_2^-): Λ_2 \dashrightarrow Λ_1^-\times Λ_2^-\big){}^2 $ equals
\begin{equation}
\left|\mathrm{Disc}_{E_1/F}\cdot\mathrm{Disc}_{E_2/F}\right|_F^{-n^0\cdot n^1}\left|\Res\left(\mathrm{Inv}(β^0), \mathrm{Inv}(β^1)\right)\right|_F^{-2}\cdot q^{2n^1\cdot|m^0| - 2n^0 \cdot|m^1|}.
\end{equation}
Taking the square root proves Prop. \ref{prop:key_degree_computation}.
\end{proof}

\begin{proof}[Proof of Propositions \ref{prop:fiber_count_analytic_E03}, \ref{prop:fiber_count_analytic_E12}, \ref{prop:fiber_count_geometric} and Theorems \ref{thm:analytic_reduction}, \ref{thm:main}.] We take up the setting of Lem. \ref{lem:phi_isogeny}. Recall that this lemma concerned the quasi-homomorphism
\begin{equation}\label{eq:recap_pp}
\m{p_2^-}, {\wt{p}_1^-}.: \Hom(T_0, Y_0) \dashrightarrow \Hom_{O_{E_2}-\text{conj}}(T_0, Y_0) \times \Hom_{O_{E_1}-\text{conj}}(T_r, Y_r).
\end{equation}
Composing with the quasi-isogenies $ρ_0^0$, $ρ_0^1$, $ρ_r^0$ and $ρ_r^1$, and applying the Dieudonné module functor $M(-)$ over the special fiber, \eqref{eq:recap_pp} is sent to $(q_1^-, q_2^-):Λ_2\dashrightarrow Λ_1^-\times Λ_2^-$. Thus Prop. \ref{prop:key_degree_computation} shows that \eqref{eq:recap_pp} is a quasi-isogeny and that its degree is given by \eqref{eq:yaoling}. Then Lem. \ref{lem:phi_isogeny} states that $Φ$ is an isogeny. Cearly,
$$\prod_{i = 1}^r \deg(L_i) = q^{2n^0|m^1|}.$$
Combining \eqref{eq:deg_relation} with \eqref{eq:yaoling}, we obtain
\begin{equation}
\deg(Φ) = \left|\mathrm{Disc}_{E_1/F}\cdot\mathrm{Disc}_{E_2/F}\right|_F^{-\frac{n^0\cdot n^1}2}\left|\Res\left(\mathrm{Inv}(β^0), \mathrm{Inv}(β^1)\right)\right|_F^{-1}\cdot q^{n^1\cdot|m^0| + n^0 \cdot|m^1|}.
\end{equation}
This finishes the proof of all our results.
\end{proof}

\section{Appendix I: Local intersection theory}
\label{s:appendix_loc_inter}
This section provides some background on intersection theory on regular local formal schemes. Its first aim is to give a precise definition of cycles, correspondences and the action of the latter on the former. Its second aim is to isolate some cases in which intersection numbers may be computed as lengths. These allow us to work with our non-standard definition of Hecke correspondences in \S3 and \S4 of the paper.

The most important input in the following is Serre's Vanishing Conjecture Thm. \ref{thm:Serre_vanishing}, which is due (independently) to Roberts \cite{Roberts} and Gillet--Soulé \cite{GS}. In fact, we present the theory of cycles (in our context) as a special case of \cite{GS}, which thus becomes our main reference. None of the statements that follow is new.

\subsection{Cycle intersection}

Let $M = \Spf A$ be the formal spectrum of a regular complete local ring of dimension $n$. We write $M^{(c)}$ for the set of integral closed formal subschemes $Z\subseteq M$ of codimension $c$. These are the same as integral closed subschemes of $\Spec A$ and we write $η_Z \in \Spec A$ for the generic point of $Z$. Denote by
\begin{equation}\label{eq:def_cycle_group}
Z^c(M) := \bigoplus_{Z\in M^{(c)}} \mbZ\cdot [Z]
\end{equation}
the group of cycles of codimension $c$. The next definition is the standard way of defining cycles from coherent sheaves: Take the maximal dimensional irreducible components of its support with appropriate multiplicties.
\begin{defn}\label{def:cycle_for_complex}
(1) For a finite type $A$-module $E$ with $\codim \Supp(E) \geq c$, we define
\begin{equation}\label{eq:cycle_for_coherent}
[E]^c := \sum_{Z\in M^{(c)}} \ell_{A_{η_Z}}\left(E_{η_Z}\right)\cdot [Z] \in Z^c(M).
\end{equation}
Here, $\ell$ is our notation for the length of an artinian module, while $A_{η_Z}$ and $E_{η_Z}$ denote the localizations at $η_Z$.\\
(2) Given a bounded complex of finite type $A$-modules $E^\bullet$, we write $\Supp(E^\bullet) := \bigcup_{i\in \mbZ} \Supp(H^i(E^\bullet))$. If $\codim \Supp(E^\bullet) \geq c$, we define
\begin{equation}\label{eq:cycle_for_complex}
[E^\bullet]^c := \sum_{i\in \mbZ} (-1)^i \left[H^i(E^\bullet)\right]^c \in Z^c(M).
\end{equation}
\end{defn}
In fact, every such complex $E^\bullet$ is quasi-isomorphic to a perfect complex because $A$ is regular. (A complex is perfect if it is bounded and if all its terms are finite free.) For $Y\subseteq \Spec A$ closed, let $K_0^Y(M)$ denote the $K$-group of perfect complexes acyclic outside $Y$, cf. \cite[\S1]{GS}. Let $K_0'(Y)$ denote the $K$-group of coherent sheaves with support on $Y$. We write $(E^\bullet)$ (resp. $(E)$) for the class of a perfect complex (resp. coherent sheaf) in these $K$-groups. Then, by \cite[Lem. 1.9]{GS},
\begin{equation}\label{eq:iso_K_group}
\begin{aligned}
K_0^Y(M) & \overset{\iso}{\longrightarrow} K'_0(Y)\\
(E^\bullet) & \longmapsto \sum_{i\in \mbZ} (-1)^i (H^i(E^\bullet)).
\end{aligned}
\end{equation}
It is clear that \eqref{eq:cycle_for_complex} descends to a map $K'_0(Y) \to Z^c(M)$ whenever $\codim Y \geq c$ and that \eqref{eq:cycle_for_coherent} then comes from composition with \eqref{eq:iso_K_group}. There is an associative, commutative cup product
\begin{equation}
\label{eq:cup_K_groups}
\begin{aligned}
\cup:K_0^Y(M)\times K_0^Z(M) & \longrightarrow K_0^{Y\cap Z}(M)\\
(E_1^\bullet), (E_2^\bullet) & \longmapsto (E_1^\bullet \overset{\mathbb{L}}{\tensor}_{\mcO_M} E_2^\bullet).
\end{aligned}
\end{equation}
Every cycle $z\in Z^c(M)$ may be viewed as an element of $K'_0(\Supp z)$ by linear extension from the map $M^{(c)}\ni Z\mapsto \mcO_Z$. Applying the cup product and composing with the cycle map of Def. \ref{def:cycle_for_complex} provides an associative, commutative intersection product of cycles whenever intersections are non-degenerate.
\begin{defn}\label{def:cycle_intersection_product}
Given $Z_1 \in M^{(c_1)}$ and $Z_2\in M^{(c_2)}$ such that $\codim (Z_1\cap Z_2) \geq c_1 + c_2$, put
\begin{equation}\label{eq:derived_intersection}
[Z_1]\cdot [Z_2] := \left[\mcO_{Z_1}\overset{\mathbb{L}}{\tensor}_{\mcO_M} \mcO_{Z_2}\right]^{c_1+c_2} \in Z^{c_1+c_2}(M).
\end{equation}
This is feasible because, as mentioned before, the $\mcO_{Z_i}$ are quasi-isomorphic to perfect complexes. Write $(Z^{c_1}(M)\times Z^{c_2}(M))_{\mr{proper}}$ for the set of those pairs $(z_1,z_2)$ with $\codim (\Supp (z_1) \cap \Supp (z_2)) \geq c_1 + c_2$. Then \eqref{eq:derived_intersection} extends linearly to a map
$$(Z^{c_1}(M)\times Z^{c_2}(M))_{\mr{proper}}\longrightarrow Z^{c_1+c_2}(M).$$
\end{defn}
\begin{rmk}\label{rmk:derived_intersection}
Formula \eqref{eq:derived_intersection} may be written more concretely as
\begin{equation}\label{eq:Serre_intersection}
[Z_1] \cdot [Z_2] := \sum_{Z\in M^{(c_1+c_2)}}\sum_{i\geq 0} (-1)^i\ell_{A_{η_Z}} \left(\Tor_i^{A_{η_Z}}(\mcO_{Z_1,η_Z}, \mcO_{Z_2,η_Z}) \right)\cdot [Z] \in Z^{c_1+c_2}(M).
\end{equation}
\end{rmk}
\begin{thm}[Serre's Vanishing Conjecture, \protect{\cite[Thm. 5.4]{GS}}]\label{thm:Serre_vanishing}
Let $E_1^\bullet$ and $E_2^\bullet$ be two perfect complexes of $A$-modules with $\codim \Supp(E^\bullet_1) + \codim\Supp(E^\bullet_2) > n = \dim A$. Then
$$[E^\bullet_1 \overset{\mathbb{L}}{\tensor} E^\bullet_2]^n = 0.$$
\end{thm}
\begin{cor}\label{cor:commutativity_taking_cycles}
Let $E_1^\bullet$ and $E_2^\bullet$ be two perfect complexes of $A$-modules with $\codim \Supp(E_i^\bullet)\geq c_i$. Further assume $\codim (\Supp(E_1^\bullet) \cap \Supp(E_2^\bullet)) \geq c_1 + c_2$. Then
\begin{equation}\label{eq:comp_inter_defs}
[E^\bullet_1]^{c_1}\cdot [E^\bullet_2]^{c_2} = [E^\bullet_1 \overset{\mathbb{L}}{\tensor}_{\mcO_M} E^\bullet_2]^{c_1+c_2}.
\end{equation}
\end{cor}
\begin{proof}
The left hand side of \eqref{eq:comp_inter_defs} is defined by viewing $[E_1^\bullet]^{c_1}$ and $[E_2^\bullet]^{c_2}$ as elements of $K'_0(\Supp(E_i^\bullet))$ by linear extension of the map $M^{(c)}\ni Z \mapsto \mcO_Z$. In $K'_0(\Supp(E_i^\bullet))$ we may write
$$[E_i^\bullet]^{c_i} = (E_i^\bullet) + ε_i$$
for some coherent sheaf $ε_i$ with $\codim \Supp ε_i > c_i$. Applying Thm. \ref{thm:Serre_vanishing} to the localizations $A_{η_Z}$ for $Z\in M^{(c_1+c_2)}$ we obtain
$$[ε_1 \overset{\mathbb{L}}{\tensor}_{\mcO_M} E_2^\bullet]^{c_1+c_2} = [ E_1^\bullet \overset{\mathbb{L}}{\tensor}_{\mcO_M} ε_2]^{c_1+c_2} = [ε_1 \overset{\mathbb{L}}{\tensor}_{\mcO_M} ε_2]^{c_1+c_2} = 0.$$
\end{proof}

\subsection{Correspondences}
Let $W$ denote a complete DVR. Assume from now on that $M$ is local, formally of finite type, and formally smooth of relative dimension $n-1$ over $\Spf W$. We define an intersection number as follows.
\begin{defn}\label{def:intersection_number}
The degree $\mr{deg}:Z^n(M) \overset{\sim}{\longrightarrow} \mbZ$ is defined as $\deg(\mcO_Z) = \ell_{W}(\mcO_Z)$. Composing with the intersection product \eqref{eq:derived_intersection} defines the intersection number of properly intersecting cycles,
\begin{equation}\label{eq:def_int_number}
(\ ,\ ):(Z^{c}(M) \times Z^{n-c}(M))_{\mr{proper}} \longrightarrow \mbZ,\ \ (z_1,z_2) = \mr{deg}(z_1\cdot z_2).
\end{equation}
\end{defn}
The formal smoothness provides that $M\times_{\Spf W} M$ is again regular and of dimension $2n-1$. Consider the diagram
$$
\xymatrix{
M \times_{\Spf W} M \ar[r]^-{p_2} \ar[d]_{p_1} & M\\
M.
}
$$
There is a pullback of cycles
$$p_2^*:Z^c(M)\longrightarrow Z^c(M \times_{\Spf W} M),\ \ \ [Z]\longmapsto [p_2^{-1}(Z)].$$
There is also a pushforward of cycle whose support is finite (via $p_1$) over $M$,
$$p_{1,*}:Z^c(M\times_{\Spf W} M)_{\text{$p_1$-finite}} \longrightarrow Z^{c-n+1}(M),\ \ \ [Z]\longmapsto \deg(Z\to p_1(Z))[p_1(Z)].$$
These considerations similarly apply to the various projection maps
$$p_{ij}:M\times_{\Spf W} M \times_{\Spf W} M \longrightarrow M\times_{\Spf W} M.$$
\begin{defn}\label{def:correspondences}
(1) Denote by $\mr{Corr}(M)\subseteq Z^{n-1}(M\times_{\Spf W}M)$ all those elements whose support is finite over $M$ via both $p_1$ and $p_2$. Define the bilinear map
\begin{equation}\label{eq:def_ring_str_correspondences}
\begin{aligned}
*:\mr{Corr}(M)\times \mr{Corr}(M) & \longrightarrow \mr{Corr}(M)\\
(C_1, C_2) & \longmapsto C_1*C_2 := p_{13,*}\left(p_{12}^*(C_1)\cdot p_{23}^*(C_2)\right).
\end{aligned}
\end{equation}
Note that the intersection $\Supp p_{12}^*(C_1) \cap \Supp p_{23}^*(C_2)$ is finite over $M$ (via both $p_1$ and $p_3$) and hence has the expected codimension.\\
(2) Define the bilinear map
\begin{equation}\label{eq:def_action_correspondences}
\begin{aligned}
\mr{Corr}(M)\times Z^c(M) & \longrightarrow Z^c(M)\\
(C,z) & \longmapsto C*z := p_{1,*}(C\cdot p_2^*(z)).
\end{aligned}
\end{equation}
Again, the occurring intersection is of the expected codimension and finite over $M$ (via both $p_1$ and $p_2$.)
\end{defn}
Let $R$ be a formal scheme, let $p_1,p_2:R\to M$ be two finite maps and assume that $R$ is pure of dimension $n$. Then also $(p_1,p_2):R\to M\times_{\Spf W}M$ is finite and $R$ defines the correspondence $[R] := [\mcO_R]^{n-1}$.
\begin{prop}\label{prop:correspondences}
(1) Definitions \eqref{eq:def_ring_str_correspondences} and \eqref{eq:def_action_correspondences} make $\mr{Corr}(M)$ into a ring and $Z^c(M)$ into a left $\mr{Corr}(M)$-module.\\
(2) Moreover, if $C_1 = [R_1]$ and $C_2 = [R_2]$ for finite maps $R_i \to M\times_{\Spf W} M$ from purely $n$-dimensional formal schemes $R_i$, then simply
\begin{equation}\label{eq:ident_Hecke}
C_1 * C_2 = [R_1 \times_{p_2,\,M,\,p_1} R_2].
\end{equation}
(3) Similarly, if $C \in \mr{Corr}(M)$ has the form $C = [R]$ and $z = [Z] \in Z^c(M)$, then
\begin{equation}\label{eq:ident_Hecke_action}
C*z = p_{1,*}[R\times_{p_2,M} Z].
\end{equation}
\end{prop}
\begin{proof}
Statement (1) claims associativity of \eqref{eq:def_ring_str_correspondences} and \eqref{eq:def_action_correspondences} Using the projection formula \cite[(1.7.2)]{GS}, this reduces to associativity of the cycle intersection for
$$p_{12}^*C_1\cdot p_{23}^*C_2\cdot p_{34}^*C_3,\ \ \ \mr{resp.}\ \ \ p_{12}^*C_1\cdot p_{23}^*C_2 \cdot p_3^*z$$
on $M^4$ resp. $M^3$, where $C_i\in \mr{Corr}(M)$ and $z\in Z^c(M)$. For example, the first expression may be rewritten as
$$\begin{aligned}
p_{14,*}(p_{12}^*C_1\cdot p_{23}^*C_2\cdot p_{34}^*C_3) & = p_{13,*}p_{124,*}(p_{124}^*(p_{12}^*C_1) \cdot p_{23}^*C_2\cdot p_{34}^*C_3) \\
& = p_{13,*}(p_{12}^*C_1\cdot p_{124,*}(p_{23}^*C_2 \cdot p_{34}^*C_3)) \\
& = p_{13,*}(p_{12}^*C_1 \cdot p_{23}^*(C_2*C_3)) \\
& = C_1 * (C_2*C_3).
\end{aligned}$$
We omit further details.

Let now $R_1,R_2 \to M\times_{\Spf W} M$ be as in Statement (2). By definitions of the two sides in \eqref{eq:ident_Hecke}, we need to see the equality
\begin{equation}\label{eq:int_identity_Hecke_corrs}
p_{13,*}[\mcO_{p_{12}^{-1}(R_1)} \overset{\mathbb{L}}{\tensor}_{\mcO_{M\times M\times M}} \mcO_{p_{23}^{-1}(R_2)}]^{2n-2} = p_{13,*}[\mcO_{R_1} \tensor_{p_2^*, \ \mcO_M,\ p_1^*} \mcO_{R_2}]^{2n-2}.
\end{equation}
In fact, we claim this before applying $p_{13,*}$ and we show the slightly stronger statement that all higher Tor-terms on the left hand side of \eqref{eq:int_identity_Hecke_corrs} have support in codimension $\geq 2n-1$. To see this, replace both $\mcO_{R_i}$ by perfect complexes $E_i^\bullet$ on $M\times_{\Spf W} M$. Then there is an isomorphism of the following kind, which may be checked term by term,
$$(E^\bullet_1 \tensor_W A) \tensor_{A \tensor_WA \tensor_WA} (A \tensor_W E^\bullet_2) \iso (E^\bullet_1 \tensor_{p_2^*,\,A,\,p_1^*}E^\bullet_2).$$
The maps $p_1^*, p_2^*:A \to A\tensor_W A$ are flat, so the $E_i^\bullet$ provide flat resolutions of $\mcO_{R_i}$ as $A$-modules. Thus we find that
\begin{equation}\label{eq:tor_computation}
\mcO_{p_{12}^{-1}R_1} \overset{\mathbb{L}}{\tensor}_{\mcO_{M\times M\times M}} \mcO_{p_{23}^{-1}R_2} \iso \mcO_{R_1}\overset{\mathbb{L}}{\tensor}_{p_2^*,\, \mcO_M,\, p_1^*} \mcO_{R_2}.
\end{equation}
But $A$ is an integral domain, so the various projection maps $R_1,R_2\to M$ are generically flat. Thus the higher Tor-terms on the right hand side of \eqref{eq:tor_computation} have support in dimension $< \dim M$ as claimed, proving (2). The proof of (3) is identical and omitted.
\end{proof}

We finally recall that intersection numbers may often be computed as lengths. This relies on the following fact about Cohen--Macaulay rings.
\begin{lem}[\protect{\cite[Tag 02JN]{Stacks}}]\label{lem:Cohen_Macaulay_simple}
Let $R$ be a Cohen--Macaulay noetherian local ring. A sequence of elements $x_1,\ldots,x_r\in R$ is regular if and only if $\dim (R/(x_1,\ldots,x_r)) = \dim(R) - r.$
\end{lem}
\begin{cor}\label{cor:cycle_resolution_intersection_number}
(1) Let $Z_1, Z_2\subseteq M$ be closed formal subschemes. Assume that $Z_1$ is Cohen--Macaulay and that $Z_2$ is defined by a regular sequence. Assume further that $\codim Z_1 + \codim Z_2 = n$ and that $Z_1\cap Z_2$ is artinian. Then
$$([Z_1], [Z_2]) = \ell_W(\mcO_{Z_1 \cap Z_2}).$$

(2) Assume that $R$ is Cohen--Macaulay of pure dimension $n$ and that $p_1,p_2:R\to M$ are two finite maps. Let $Z_1,Z_2\subseteq M$ be closed formal subschemes that are defined by regular sequences and such that $\codim Z_1 + \codim Z_2 = n$. Assume further that $R\times_{M\times_{\Spf W} M}(Z_1\times_{\Spf W} Z_2)$ is artinian. Then
$$([Z_1], [R]*[Z_2]) = \ell_W\left(R\times_{M\times_{\Spf W} M}(Z_1\times_{\Spf W} Z_2)\right).$$
\end{cor}
\begin{proof}
(1) This is standard. Assume  $Z_2 = V(x_1,\ldots,x_c)$ for a regular sequence $x_1,\ldots,x_c$. By Lem. \ref{cor:cycle_resolution_intersection_number} and our dimension assumptions, $x_1\vert_{Z_1},\ldots,x_c\vert_{Z_1}$ define a regular sequence on $Z_1$. Resolving $\mcO_{Z_2}$ by the Koszul complex of $(x_1,\ldots,x_1)$ shows that $\Tor_i^{\mcO_M}(\mcO_{Z_1}, \mcO_{Z_2}) = 0$ for $i > 0$. Thus $\mcO_{Z_1\cap Z_2}$ represents $\mcO_{Z_1}\tensor^{\mathbb{L}}_{\mcO_M} \mcO_{Z_2}$. By Cor. \ref{cor:commutativity_taking_cycles}, $\ell_W(\mcO_{Z_1 \cap Z_2})$ then agrees with $([Z_1], [Z_2])$.

(2) By the projection formula \cite[(1.7.2)]{GS} or by Prop. \ref{prop:correspondences},
$$([Z_1], [R]*[Z_2]) = \deg([R] \cdot p_1^*[Z_1] \cdot p_2^*[Z_2]).$$
Cor. \ref{cor:commutativity_taking_cycles} provides the first equality in
$$p_1^*[Z_1] \cdot p_2^*[Z_2] = [\mcO_{p_1^{-1}(Z_1)} \overset{\mathbb{L}}{\tensor}_{\mcO_{M\times M}} \mcO_{p_2^{-1}(Z_2)}]^n = [Z_1\times_{\Spf W} Z_2]^n.$$
The second equality follows because $p_1^{-1}(Z_1)$ and $p_2^{-1}(Z_2)$ are Tor-independent. Now we may apply Statement (1) for the formal scheme $M\times_{\Spf W} M$, finishing the proof.
\end{proof}

\section{Appendix II: Partial Satake transformation for $GL_n$}
\label{s:appendix_Satake}

\subsection{Partial Satake transformation}

Let $G = GL_n$ and let $P=LU\subset G$ be a parabolic with Levi factor $L$ and unipotent radical $U$. Let $\delta_{L}$ be the modular character of $P$, meaning that $d(lul^{-1}) = \delta_{L}(l) du$ for $u\in U$, $l\in L$ and every left Haar measure $du$ on $U$. Up to conjugacy, the situation is that of a standard parabolic: Let $n_\bullet=(n_1,\cdots,n_k)$ be a partition of $n$ and let $P_{n_\bullet}$ be the block upper triangular matrices with block size given by $n_\bullet$. Let $L_{n_\bullet}$ be the Levi of block-diagonal matrices and let $U_{n_\bullet}$ be the unipotent radical. For $l=(l_1,l_2,\cdots,l_k)\in L_{n_\bullet}$ with $l_i\in\GL_{n_i}(F)$, the modular character is given by
\begin{equation}\label{eq:modular_character}
\delta_{L_{n_\bullet}}(l)=\prod_{i=1}^k |\det l_i|_F^{-n_1-\cdots-n_{i-1}+n_{i+1}+\cdots+n_k}.
\end{equation}
The $\mathbb C$-coefficient (partial) Satake transformation is the following map of spherical Hecke algebras 
\begin{equation}\label{eq:part_Satake}
\begin{aligned}
\mathcal S:\mathbb C\left[G(O_F)\backslash G(F)/G(O_F)\right]
\longrightarrow 
\mathbb C\left[L(O_F)\backslash L(F)/L(O_F)\right]\\
f \longmapsto \mathcal Sf;\quad \mathcal Sf(l):=\delta_{L}(l)^{\frac12}
\int_{U(F)}f\left(lu\right)du
\end{aligned}
\end{equation}
where $du$ is the Haar measure normalized by $U(O_F)$. 
\begin{lem}\label{lem:Satake_ring_hom}
The partial Satake transformation \eqref{eq:part_Satake} is a map of $\mbC$-algebras.
\end{lem}
\begin{proof}
The integral
$$
\mathcal S[f*g](l)=\delta_L(l)^{\frac12}\cdot \int_{U(F)}\int_{G(F)} f(y^{-1}lu)g(y)dydu
$$
is invariant under $y\mapsto yk$ for $k\in G(O_F)$. The Iwasawa decomposition provides a $P$-equivariant isomorphism
$$
P(F)/P(O_F)\cong G(F)/G(O_F).$$
The measure on the two spaces here is given by the left invariant Haar measure $dl\,du$ on $P = LU$. Therefore
$$
\mathcal S[f*g](l)=\delta_L(l_2^{-1}l)^{\frac12}\delta_L(l_2)^{\frac12}\int_{U(F)}\int_{U(F)}\int_{L(F)} f(u_2^{-1}l_2^{-1}lu)g(l_2u_2)dudu_2d l_2.
$$
Applying the change of variables $u\longmapsto l^{-1}l_2u_2l_2^{-1}lu$, we conclude that 
$$\mathcal S[f*g](l)=\int_{L(F)}\left(\underbrace{\delta_L(l_2^{-1}l)^{\frac12}\int_U f(l_2^{-1}lu)\;du}_{\mathcal S[f](l_2^{-1}l)}\cdot \underbrace{\delta_L(l_2)^{\frac12}\int_U g(l_2u_2)\;du_2}_{\mathcal S[g](l_2)}\right)d l_2
=(\mathcal S[f]*\mathcal S[g])(l)$$
as desired.\end{proof}
We next specialize to the usual Satake transformation, which is the case of the partition $n=1+1+\cdots+1$. We write $T$ for the diagonal torus and $U$ for the unipotent radical, now consisting of upper triangular unipotent matrices. Since $T$ is abelian, we have $T(O_F)\backslash T(F)/T(O_F)=T(F)/T(O_F)$.

Denote by $x_i: T\longrightarrow \{0,1\}$ the characteristic function of the subset 
$$O_F^\times\times\cdots O_F^\times\times\underbrace{\pi O_F^\times}_{i\text{-th component}}\times  O_F^\times\cdots\times O_F^\times\subset T.$$
Then each $x_i$ is invertible in $\mbC[T(F)/T(O_F)]$ and there is the explicit description
\begin{equation}\label{polynomials}
\mathbb C\left[T(F)/T(O_F)\right]= \mathbb C[x^{\pm}_1,\cdots, x^{\pm}_n].
\end{equation}
The monomial $x_1^{m_1}\cdots x_n^{m_n}$ on the right hand side is simply the characteristic fuction of the subset $\pi^{m_1} O_F^\times\times \pi^{m_2} O_F^\times \times\cdots\times \pi^{m_n} O_F^\times$.

Put $\mathrm{Mat}_{n\times n}^{\det\neq 0}(O_F):=\mathrm{Mat}_{n\times n}(O_F)\cap \GL_n(F)$. Let $W=N(T)/T$ be the Weyl group. For any two $ O_F$-modules $\Lambda_1$ and $\Lambda_2$, we write $\Lambda_1\underset k\subset \Lambda_2$ to denote a (not necessarily strict) inclusion $\Lambda_1\subseteq \Lambda_2$ with $|\Lambda_2/\Lambda_1|=q^k$, where $q=|O_F/\pi O_F|$.
\begin{defn}\label{operators}
For any $0 \leq k\leq n$, resp. for $0 \leq k$, define the following functions on $G(F)$:
\begin{equation}\label{eq:def_generators}
\mathbf S_k(g)=\begin{cases}
1& \text{if } \pi O_F^n\underset{n-k}\subset g O_F^n\underset k\subset  O_F^n;\\
0& \text{otherwise},
\end{cases}
\qquad\ \ 
\mathbf T_k(g)=\begin{cases}
1& \text{if } g O_F^n\underset k\subset  O_F^n; \\
0& \text{otherwise}.
\end{cases}
\end{equation}
For any $-n\leq k<0$, resp. for $k<0$, put $\mathbf S_k(g):=\mathbf S_{-k}(g^{-1})$ and $\mathbf T_k(g):=\mathbf T_{-k}(g^{-1})$. In particular $\mathbf S_n$, $\mathbf S_0$ and $\mathbf S_{-n}$ are the characteristic functions of $\pi G(O_F)$, $G(O_F)$ and $\pi^{-1}G(O_F)$, respectively. For example, $(\mathbf S_{-n}*f)(g) = f(πg)$ for every Hecke function $f$.
\end{defn}

Our first result in this appendix is the following theorem.
\begin{thm}\label{mubiao}
Let $G^{(k)}$ be the preimage of $\pi^k O_F^\times$ under the map $\det: \mathrm{Mat}_{n\times n}(O_F)\longrightarrow O_F$ and let
\begin{equation}\label{H}
\mathcal H^{(k)}:=\mathbb C\left[G(O_F)\backslash G^{(k)}/G(O_F)\right], \qquad
\mathcal H:=\bigoplus_{k=0}^\infty \mathcal H^{(k)}.
\end{equation}
The following statements hold.
\begin{enumerate}[wide, labelindent=0pt, labelwidth=!, label=(\arabic*), topsep=2pt, itemsep=2pt]
\item The Satake transformation defines an isomorphism from $\mathcal H^{(k)}$ to the vector space of degree $k$ homogeneous symmetric polynomials, and consequently from $\mathcal H$ to the algebra of symmetric polynomials
\begin{equation}\label{eq:Satake_usual}
\mathcal S : \mathcal H^{(k)}\overset\cong\longrightarrow \mathbb C[x_1,\cdots,x_n]^{W=\mathrm{id},\ \deg=k};\qquad\ \ 
\mathcal S : \mathcal H\overset\cong\longrightarrow \mathbb C[x_1,\cdots,x_n]^{W=\mathrm{id}}.
\end{equation}
\item Either set $\{\mathbf S_i\}_{0\leq i\leq n}$ or $\{\mathbf T_i\}_{0\leq i\leq n}$ generates $\mathcal H$ as a $\mathbb C$-algebra.
\item Either set $\{\mathbf S_{-n}\}\cup \{\mathbf S_i\}_{0\leq i\leq n}$ or $\{\mathbf S_{-n}\}\cup \{\mathbf T_i\}_{0\leq i\leq n}$ generates $\mathbb C[G(O_F)\backslash G(F)/G(O_F)]$ as a $\mathbb C$-algebra.
\end{enumerate}
\end{thm}

Note that (3) follows immediately from (2). The proof of (1) and (2) will be given in the next few sections and will be complete after \S\ref{ss:proof_end}. Part (1) is actually well-known and may be found in more general form in \cite{Gross}, but we include a proof for convenience. We begin here with the observation that it is sufficient to prove surjectivity in (1):
\begin{lem}\label{lem:surjectivity_lemma}
Assume $\mcS$ maps $\mcH$ surjectively onto $\mbC[x_1,\ldots,x_n]^{W = \mathrm{id}}$. Then $\mcS:\mcH\iso \mbC[x_1,\ldots,x_n]^{W = \mathrm{id}}$ is bijective.
\end{lem}
\begin{proof}
Since $\det(g\cdot U)=\det(g)$, we have $\mathcal S(\mathcal H^{(k)})\subset \mathbb C[x_1,\cdots,x_n]^{\deg=k}$ for any integer $k$. By counting double cosets via the Cartan decomposition, 
\begin{multline*}
\#\left(G(O_F)\backslash G^{(k)}/G(O_F)\right)=\#\{k_1\leq k_2\leq\cdots\leq k_n: k_1+k_2+\cdots+k_n=k\}\\
=\dim \mathbb C[x_1,\cdots,x_n]^{W = \mr{id},\deg=k},
\end{multline*}
we conclude that the $\mathbb C$-vector spaces $\mathcal H^{(k)}$ and $\mathbb C[x_1,\cdots,x_n]^{W = \mathrm{id},\deg=k}$ are of the same dimension and the lemma follows.
\end{proof}

\subsection{Satake transformation in terms of orbital integrals}
In this section we prove the inclusion
\begin{equation}\label{Weyl}
\mathcal S \left(\mathcal H\right) \subset \mathbb C[x_1,\cdots,x_n]^{W=\mathrm{id}}.
\end{equation}

Let $U_{n-1}:= U$ and inductively define $U_{i-1}:=[U_i,U_i]$ as the derived subgroup, meaning the minimal normal subgroup such that $U_i/U_{i-1}$ is abelian. Explicitly,  
$$U_i=\left\{(a_{ij})_{n\times n}:a_{k,k}=1\text{ and }a_{k,k+j}=0 \text{ for } j<0 \text{ and for }0<j<n-i\right\}.$$
 Then 
$$U_i/U_{i-1}\cong\{(a_{1,1+n-i},a_{2,2+n-i},\cdots,a_{i,n})\}\cong  F^{\oplus i}$$
 and for any $t=\mathrm{diag}(t_1,\cdots,t_n)\in T$, the map 
$$
U\longrightarrow U;\qquad   u\longmapsto  t^{-1} u tu^{-1}
$$
induces linear transformations on $U_i/U_{i-1}$ with eigenvalues $\{t_k^{-1}t_{k+n-i}-1\}_{1\leq k\leq i}$. Assume that the $t_i$ are pairwise different. Then, denoting by $du$ the Haar measure on $U$, this implies that
$$
d(t^{-1}utu^{-1}) = \prod_{i=1}^{n-1}\prod_{k=1}^i |t_k^{-1}t_{k+n-i}-1|_F\cdot du=|\det t|_F^{-\frac{n-1}2}\cdot \delta_T(t)^{-\frac 12}\prod_{1\leq i<j\leq n}\left|t_i-t_j\right|_F du.
$$
Note that the product of $|t_i-t_j|_F$ for $1\leq i<j\leq n$ equals to $\left|\mathrm{Disc}(\mathrm{char}_t)\right|_F^{\frac12}$ where $\mathrm{char}_t(\lambda)=\det(\lambda\cdot\mathrm{id}-t)$ is the characteristic polynomial of $t$, and where $\mathrm{Disc}$ denotes the discriminant. We have
$$
\begin{aligned}
\mathcal S[f](t) &\ =\ \delta_T(t)^{\frac12}\int_{U(F)} f\left(t (t^{-1}utu^{-1})\right) d(t^{-1}utu^{-1})\\
&\ =\ |\det t|_F^{-\frac{n-1}2} \left|\mathrm{Disc}(\mathrm{char}_t)\right|_F^{\frac12}\int_{U(F)} f(utu^{-1})du.
\end{aligned}
$$
Since $\GL_n(F)=\mathrm{GL}_n(O_F) U(F) T(F)$ by the Iwasawa decomposition and since $f$ is $\GL_n(O_F)$-invariant, the occurring orbital integral is
$$
\int_{U(F)} f(utu^{-1})du = \int_{\GL_n(F)/T(F)} f(hth^{-1})dh.
$$
This expression only depends on $t$ up to conjugation. Therefore, the Satake transformation is Weyl-invariant as claimed in \eqref{Weyl}.

\subsection{Satake transformation of miniscule Hecke functions}
We next determine the images $\mcS[\mathbf S_k]$. Write $V_n=F^{n}$ and let $V_n\supset V_{n-1}\supset\cdots\supset V_0=\{0\}$ be the unique $U$-invariant filtration such that $V_i/V_{i-1}\cong F$. For any $t=\mathrm{diag}(t_1,t_2,\cdots,t_n)$ and any $u\in U$, the linear transformation $tu:V_n\longrightarrow V_n$ induces scaling by $t_i$ on $V_i/V_{i-1}$. Therefore, $O_F^{n}\supseteq tu O_F^{n}\supseteq \pi\cdot O_F^{n}$ implies that $t_i\in  O_F^\times\cup\pi\cdot O_F^\times$. This shows that
$$
\mathcal S[\mathbf S_k]=\sum_{\substack{I\subset \{1,\cdots,n\}\\\#I=k}}C_I\cdot\prod_{i\in I}x_i
$$ 
for some constants $C_I$. The polynomial on the right hand side is $W$-invariant by \eqref{Weyl}, so the constants $C_I$ are actually independent of $I$. 

Now we determine $C_{\{n-k+1,\cdots,n\}}$. Consider the element $t=\mathrm{diag}(1,1,\cdots,1,\pi,\cdots\pi)$ and note that $\mathbf S_k(tu)=1$ if and only if $u\in U(O_F)$. Therefore $\mathcal S[\mathbf S_k](t)=\delta_T(t)^{\frac 12}=q^{\frac{kn-k^2}2}$ which implies that
$$
\mathcal S[\mathbf S_k]=q^{\frac{kn-k^2}2}\sum_{\substack{I\subset \{1,\cdots,n\}\\\#I=k}}\prod_{i\in I}x_i
$$
is (a non-zero scalar multiple of) an elementary symmetric polynomial. It is well-known that these generate $\mathbb C[x_1,\cdots,x_n]^{W=\mathrm{id}}$, so
$$
\mathcal S (\mcH) = \mathbb C[x_1,\cdots,x_n]^{W=\mathrm{id}}.
$$
At this point, we have shown (2) of Thm. \ref{mubiao} for the set $\{\mathbf S_k\}_{0\leq k \leq n}$.

\subsection{Generating symmetric polynomials}
The next result provides an alternative generating set for the algebra of symmetric polynomials. Define
\begin{equation}\label{generators}
s_k(x_1,\cdots,x_n):=\sum_{\substack{I\subset \{1,2,\cdots,n\}\\\#I=k}}\prod_{i\in I}x_i\qquad \text{and}\qquad b_m(x_1,\cdots,x_n):=\sum_{\substack{m=m_1+\cdots+m_n\\m_1,\cdots,m_n\geq 0}}\prod_{i=1}^n x_i^{m_i}.
\end{equation}
It was already said and used that $\{s_k\}_{0\leq k\leq n}$ is a set of generators of $\mathbb C[x_1,\cdots,x_n]^{W=\mathrm{id}}$.

\begin{prop}\label{shengcheng}
The set $\{b_k(x_1,\cdots,x_n)\}_{0\leq k\leq n}$ generates $\mathbb C[x_1,\cdots,x_n]^{W=\mathrm{id}}$ as a $\mathbb C$-algebra.
\end{prop}
\begin{proof}
Note that $b_0=s_0=1$. We claim that, for all $k\geq 1$,
\begin{equation}
\sum_{i=0}^k (-1)^ib_is_{k-i} = 0.
\end{equation}
To prove this, consider the polynomial expansion
$$
\prod_{i=1}^n(X-x_i) = \sum_{i=0}^n (-1)^i s_i X^{n-i}.
$$
and that of its reciprocal
$$
\prod_{i=1}^n(X-x_i)^{-1} = \sum_{i=0}^\infty b_i X^{-i-n}.
$$
The claim follows by multiplying these two equations together. The proposition now follows by induction from the identities $s_k=-\sum_{i=1}^{k}(-1)^ib_is_{k-i}$, $k\geq 1$.
\end{proof}

\subsection{Satake transformation of $\mathbf T_m$}
\label{ss:proof_end}
\begin{prop}\label{prop:Satake_T}
We have 
$$
\mathcal S[\mathbf T_m] =q^{\frac{m(n-1)}2} b_m(x_1,\cdots,x_n).
$$
\end{prop}
\begin{proof}
Consider a tuple $m_\bullet = (m_1,\ldots,m_n)$ of integers $m_i\geq 0$ and put $|m_\bullet|=m_1+\cdots+m_n$. Let
$$
t=\mathrm{diag}(\pi^{m_1},\pi^{m_2},\cdots,\pi^{m_n})\in T,\qquad u=(a_{ij})_{1\leq i,j\leq n}\in U.
$$
We have
$$\det(tu)=\det(t)=\pi^{|m_\bullet|}.$$
Note that $u\in U$ implies that $a_{ij}=0$ if $j<i$ and $a_{ii}=1$ for $1\leq i\leq n$. We have
$$
\begin{aligned}
\int_{U(F)} \mathbf T_m\left(
tu
\right)du
& \ =\ \mathbf 1_{\pi^m O_F^\times}\left(\pi^{|m_\bullet|}\right)\cdot\prod_{i=1}^n\prod_{j=i+1}^n \int_F \mathbf 1_{ O_F}(\pi^{m_i}a_{ij}) d a_{ij}\\
& \ =\ \mathbf 1_{\pi^m O_F^\times}\left(\pi^{|m_\bullet|}\right)\cdot\prod_{i=1}^n|\pi^{m_i}|^{-(n-i)}.
\end{aligned}
$$ 
Therefore, if $|m_\bullet| = m$, then
$$
\mathcal S[\mathbf T_m](t)=\underbrace{\prod_{i=1}^n |\pi^{m_i}|^{\frac{n+1}2-i}}_{\delta_T(t)}\cdot \prod_{i=1}^n|\pi^{m_i}|^{-(n-i)}
=\prod_{i=1}^n|\pi^{m_i}|^{\frac{1-n}2}
=|\pi^m|^{\frac{1-n}2},
$$
which implies
$$
\mathcal S[\mathbf T_m] = |\pi^m|^{\frac{1-n}2}\sum_{m=m_1+\cdots+m_n} \prod_{i=1}^nx_i^{m_i}=q^{\frac{m(n-1)}2} b_m(x_1,\cdots,x_n),
$$
where $b_m$ is the polynomial from \eqref{generators}. 
\end{proof}
Combining Prop. \ref{prop:Satake_T} with Prop. \ref{shengcheng} and the bijectivity of the Satake transformation proves that $\left\{\mathbf T_m\right\}_{0\leq m\leq n}$ is a set of generators for $\mathcal H$. The proof of Thm. \ref{mubiao} is now complete.

\newcommand{\Mat}{\mathrm{Mat}}
\subsection{Partial Satake transformation of $\mathbf T_m$}

We lastly come to the second result of this appendix, namely an explicit formula for the partial Satake transformation of $\mathbf T_m$. Suppose the partition in question is $n=n_1+n_2$ and put $g=(g_1,g_2)\in L = \GL_{n_1}(F)\times\GL_{n_2}(F)$. To avoid confusion, we write $\mathbf T_{\GL_n,m}$ for the function $\mathbf T_m$ on $GL_n(F)$.
\begin{prop}\label{explicit} For any integer $m\geq 0$, we have
$$
\mathcal S[\mathbf T_{\GL_n,m}]=\sum_{m_1+m_2=m}q^{\frac{m_1n_2+m_2n_1}2}\cdot\mathbf T_{\GL_{n_1},m_1}\boxtimes \mathbf T_{\GL_{n_2},m_2}.
$$
\end{prop}
\begin{proof}
For any
$$\mathbf u:=\m{I_{n_1}}u,{}{I_{n_2}}.\in U_{(n_1,n_2)}(F),$$
we have 
$$
\mathbf T_{\GL_n,m}(g\cdot\mathbf u)=1\iff
\begin{cases}(g_1,g_2)\in\Mat_{n_1\times n_1}(O_F)\times \Mat_{n_2\times n_2}(O_F), \\
 g_1u\in \Mat_{n_1\times n_2}(O_F), \text{ and } \\
\det g_1\cdot \det g_2\in\pi^m O_F^\times.
\end{cases}
$$
Therefore, for any $g=(g_1,g_2)\in \Mat_{n_1\times n_1}(O_F)\times \Mat_{n_2\times n_2}(O_F)$ with $\det g_1\cdot \det g_2\in\pi^m O_F^\times$, we find
$$
\mathcal S(\mathbf T_m)(g)=
\frac{|\det(g_1)|_F^{\frac{n_2}2}}{|\det(g_2)|_F^{\frac{n_1}2}}\cdot
\int_{\Mat_{n_1\times n_2}(F)}\mathbf 1_{\Mat_{n_1\times n_2}(O_F)}(g_1\cdot u)du=|\det(g_1)|_F^{-\frac{n_2}2}\cdot |\det(g_2)|_F^{-\frac{n_1}2}
$$ as desired.
\end{proof}

\end{document}